\documentclass[12pt]{article} 

\usepackage{lmodern}
\usepackage{url}
\usepackage{amsmath}
\usepackage{amsthm}
\usepackage{amsfonts}
\usepackage{mathrsfs}
\usepackage{stmaryrd}
\usepackage{setspace}
\usepackage{fullpage}
\usepackage{amssymb}
\usepackage{breqn}
\usepackage{enumerate}
\usepackage{bbm} 
\usepackage{authblk}
\usepackage{comment}
\usepackage{pgf,tikz}
\usepackage{graphicx}
\usepackage{thmtools}
\usetikzlibrary{decorations.pathreplacing,arrows}
\tikzstyle{vertex}=[circle,draw=black,fill=black,inner sep=0,minimum size=3pt,text=white,font=\footnotesize]

\usepackage{tocloft}

\usepackage{hyperref}
\hypersetup{
	colorlinks=true,
	linkcolor=blue,
	filecolor=magenta,      
	urlcolor=cyan,
}

\newtheorem*{thm*}{Theorem}
\newtheorem{thm}{Theorem}[section]
\newtheorem{theorem}[thm]{Theorem}

\newtheorem{lemma}[thm]{Lemma}
\newtheorem{conjecture}[thm]{Conjecture}
\newtheorem{proposition}[thm]{Proposition}
\newtheorem{prop}[thm]{Proposition}

\theoremstyle{definition}

\newtheorem{definition}[thm]{Definition}

\newcommand\cA{{\mathcal A}}
\newcommand\cB{{\mathcal B}}
\newcommand\cC{{\mathcal C}}
\newcommand\cD{{\mathcal D}}

\newcommand\cF{{\mathcal F}}
\newcommand\cG{{\mathcal G}}
\newcommand\cH{{\mathcal H}}

\newcommand\cL{{\mathcal L}}

\newcommand\cN{{\mathcal N}}
\newcommand\cP{{\mathcal P}}

\newcommand\cW{{\mathcal W}}

\newcommand\biex{\ensuremath{\mathrm{biex}}}
\newcommand\ex{\ensuremath{\mathrm{ex}}}
\newcommand\Ex{\ensuremath{\mathrm{EX}}}

\newcommand{\ignore}[1]{}

\title{Survey of
	generalized Tur\'an problems---counting subgraphs\thanks{The work of the first author was supported by the National Research, Development and Innovation Office - NKFIH under the grant KKP-133819. The work of the second author was supported by a grant from the Simons Foundation \#712036.}
}
\author{
D\'{a}niel Gerbner\thanks{Alfréd Rényi Institute of Mathematics, HUN-REN.  E-mail: \texttt{gerbner.daniel@renyi.hu}.} \qquad
Cory Palmer\thanks{Department of Mathematical Sciences, University of Montana. E-mail: \texttt{cory.palmer@umontana.edu}.} 
}

\begin{document}
\maketitle

\begin{abstract}
    For fixed graphs $H$ and $F$, the \emph{generalized Tur\'an number}
    $\ex(n,H,F)$
     is the maximum possible number of copies of a subgraph $H$ in an $n$-vertex $F$-free graph.
    This article is a survey of this extremal function whose study was initiated in an influential 2016 article by Alon and Shikhelman (\emph{J. Combin. Theory, B}, {\bf 121}, 2016).
\end{abstract}

\tableofcontents

\section{Introduction}

The {\it Tur\'an number} $\ex(n,F)$ is the maximum number of edges in an $F$-free graph on $n$ vertices. The study of this function is a central pursuit in extremal graph theory. Tur\'an's theorem~\cite{Tur41} determines its exact value when $F$ is a complete graph $K_k$ and the Erd\H os-Stone-Simonovits theorem~\cite{ErdSto46,ErdSim65} settles asymptotically the case when $F$ has chromatic number at least $3$. The monograph by Bollob\'as~\cite{Bol04} is a good source of results and a history of Tur\'an numbers and extremal graph theory, more generally.
Much effort has gone into understanding the behavior of this function for hypergraphs. A survey by Keevash~\cite{Kee11} gives an excellent overview.

The focus of the present survey is an alternative generalization.
As we may view an edge as a subgraph, it is natural to count the number of copies of some other graph $H$. Let $\mathcal{N}(H,G)$ denote the number of copies of a subgraph $H$ in a host graph $G$.
The maximum number of copies of a subgraph $H$ in an $F$-free graph on $n$ vertices is denoted
\[
\ex(n,H,F) := \max\{\mathcal{N}(H,G) \, : \, \text{$G$ is an $n$-vertex $F$-free graph} \}.
\]
Typically, we view this as a function of $n$ with $H$ and $F$ being fixed graphs. 
Note that this is a direct generalization of the standard Tur\'an number as $\ex(n,F) = \ex(n,K_2,F)$.
After several decades of sporadic results, Alon and Shikhelman~\cite{AloShi16} initiated the systematic study of the function $\ex(n,H,F)$. Since their manuscript a tremendous amount of work has been done on this function, under the moniker \emph{generalized Tur\'an problems}.

The purpose of this survey is to detail the history of these problems and list the myriad results. Because many of these results were published before the formalization of the generalized Tur\'an problem---often in the course of proving some other extremal result---we have taken an exhaustive approach to the collection literature. However, in some places, we will not always state all results completely or precisely for the sake of brevity. We invite the interested reader to check the corresponding references.

\medskip

{\bf Notation \& alphabet.} Most of the notation in this survey is standard in extremal graph theory. Key notation includes $\mathcal{N}(H,G)$ for the number of copies of a subgraph $H$ in a graph $G$ and, for $\mathcal{F}$ a family of graphs, $\ex(n,H,\mathcal{F})$ denotes the maximum number of copies of a subgraph $H$ in an $n$-vertex graph that does not have any subgraph isomorphic to any graph $F \in \mathcal{F}$. When $\mathcal{F}$ is a single graph $F$ we simply use $\ex(n,H,F) = \ex(n,H,\{F\})$. By \emph{extremal graph} we mean an $n$-vertex $F$-free graph that achieves the maximum number of copies of $H$, i.e., $\ex(n,H,F)$ (or edges $\ex(n,F)$, depending on the context). The family of extremal graphs for $\ex(n,H,F)$ is denoted $\Ex(n,H,F)$. The \emph{Tur\'an} graph $T(n,k)$ is the $n$-vertex complete $k$-partite graph whose partition classes are as close in size as possible.

The (vertex-disjoint) \emph{union} of graphs $F_1$ and $F_2$ is denoted $F_1 \cup F_2$ and the union of $t$ copies of $F$ is $tF$. For graphs $F_1$ and $F_2$, their \emph{join} $F_1+F_2$ takes their union and all edges between their respective vertex sets.
For undefined notation, see the monograph of Bollob\'as~\cite{bo-book}.


We have attempted to keep the notation consistent throughout this survey (even when it conflicts with the referenced papers). In particular, the parameter $r$ will be reserved for the `counting graph' $H$ (often its chromatic number) while $k$ is reserved for the `forbidden graph' $F$. Graph subscripts (e.g., $K_r$ and $C_5$) typically denote the number of vertices, so $P_k$ is the $k$-vertex path, i.e., the path of length $k-1$, or some other natural parameter (e.g.\ chromatic number, number of edges, etc).

\medskip

{\bf Organization.} 
The remainder of this section gives a history of generalized Tur\'an problems, especially those that predate the paper of Alon and Shikhelman, followed by an overview of results on classical Tur\'an numbers, which will serve as a comparison for the surveyed general results. 

The main part of the survey begins after, where we inventory the known bounds on $\ex(n,H,F)$ in several sections. We begin with general results, and among them we discuss the main differences between generalized Tur\'an problems and ordinary Tur\'an problems. One particular result tells when $\ex(n,H,F)=\Theta(n^{|V(H)|})$, the so-called \emph{non-degenerate case}. This case is discussed in Section~\ref{turangood}.
The \emph{degenerate case} when $\ex(n,H,F) = o(n^{|V(H)|})$ makes up the bulk of the survey and appears in Section~\ref{section-degen} where we arrange the known bounds on $\ex(n,H,F)$ in subsections according to properties of the graph $H$ to be counted. In Section~\ref{meth} we discuss methods that have been used frequently to attack generalized Tur\'an problems.
We conclude the survey with open problems in Section~\ref{open}.

\medskip

{\bf Part II.} Due to its length, this survey will be followed by a second part~\cite{GerPal25} that deals with further variants and extensions of the generalized Tur\'an problem. Topics include:

\begin{itemize}
    \item Counting multiple graphs, counting graphs whose order increases with $n$ (e.g.\ Hamiltonian cycles), forbidding graphs whose order increases with $n$, and forbidding an infinite family of graphs.

    \item Host graphs $G$ drawn from a particular family of graphs, e.g., $k$-partite, $k$-connected, regular, hypercubes, bounded independence number, planar, and random graphs.

    \item Induced, rainbow, colored, and weighted versions, and analogues in directed, vertex-ordered, edge-ordered graphs.

    \item Finally, we examine connections to other extremal problems, including Berge hypergraphs, graph property testing and exact bounds for removal lemmas, and local graph properties.

\end{itemize}

\subsection{History of generalized Tur\'an problems}\label{history}

In this subsection, we briefly describe some of the early results in the area (as well as some recent improvements and extensions). Our goal is to illustrate the rich history of these problems, despite being largely hidden even from experts.

In 1938, Erd\H os~\cite{Erd38} gave an upper bound on the number of ``cherries'' $P_3=K_{1,2}$ in a $C_4$-free bipartite graph in the course of investigating the multiplicative Sidon problem in number theory. The bound is given by the following argument. Each pair of vertices are the leaves in at most one cherry $P_3$, so there are at most $\binom{n}{2}$ cherries (where $n$ is the number of vertices). In other words, $\ex(n,P_3,C_4)\le \binom{n}{2}$. On the other hand, if the degrees of $G$ are $d_1,\dots, d_n$, then there are exactly $\sum_{i=1}^n \binom{d_i}{2}$ cherries in $G$. As the number of edges of $G$ is $\frac{1}{2}\sum_{i=1}^n {d_i}$, a simple calculation gives an upper bound on $\ex(n,C_4)$. This predates Tur\'an's theorem~\cite{Tur41}, and according to Simonovits~\cite{Sim13}, Erd\H os later felt that he ``should have invented'' Tur\'an theory at this point, but he ``failed to notice that his theorem was the root of an important and beautiful theory.'' Now we can say that he could have also invented generalized Tur\'an theory right away! There are several other instances of bounds for generalized Tur\'an problems being a lemma for some a related extremal problem.

The first generalized Tur\'an result standing on its own is the complete determination of $\ex(n,K_r,K_k)$ by Zykov~\cite{Zyk49} in 1949 (see~\cite{Zyk52} for an English translation).
This paper is famous for its beautiful proof of Tur\'an's theorem by Zykov's ``symmetrization method.'' However, it is not widely known that the paper proves the following more general result.

\begin{thm}[Zykov~\cite{Zyk49}]\label{zykov} If $r<k$, then
	the Tur\'an graph $T(n,k-1)$ is the unique $n$-vertex $K_k$-free graph with the maximum number of copies of $K_r$. 
\end{thm}
Counting copies of $K_r$ in $T(n,k-1)$ then gives:
	\begin{align}\label{zykov-bound}
	\ex(n,K_r,K_k) & = \mathcal{N}(K_r,T(n,k-1)) \nonumber \\ 
 &= \sum_{0 \leq i_1 < \cdots < i_r \leq k-2} \prod_{j=1}^r \left \lfloor \frac{n+i_j}{k-1} \right \rfloor \leq \binom{k-1}{r} \left(\frac{n}{k-1}\right)^r.    
	\end{align}
	
Zykov's theorem itself was rediscovered several times~\cite{Erd62,MooMos65,Sau71,Had76,Rom76,Dai09} (often in the complement setting where independent sets replace cliques). 
An early arXiv version of the seminal paper of Alon and Shikhelman~\cite{AloShi16} also contained a proof.
Weaker versions or variants also appear in~\cite{BruMel08,MooMos62},
without knowledge of Zykov's result. Cutler and Radcliffe~\cite{CutRad11} includes a new proof.
	
	Let us briefly sketch the symmetrization proof. It is no different from the case when counting edges. For two non-adjacent vertices $u$ and $v$, we say that we \textit{symmetrize} $u$ to $v$ if we remove all edges incident to $u$, and for every edge $vw$, we add the edge $uw$. Less formally, we replace the neighborhood of $u$ with the neighborhood of $v$. This cannot increase the order of the largest clique. Indeed, any new clique must contain $u$ and thus not $v$, but then we can replace $u$ with $v$ to find a clique of the same size that was present in the original graph. 
	
	For two non-adjacent vertices $u$ and $v$, we symmetrize $u$ to $v$ if $u$ is contained in no more copies of $K_r$ than $v$. In this way, the number of copies of $K_r$ does not decrease. We apply this symmetrization step to every pair of non-adjacent vertices with different neighborhoods. There are several different ways to show that this process ends after finitely many steps.
When this process does end, the resulting graph is complete multipartite, with at most $k-1$ parts. A simple optimization shows that among the complete multipartite graphs, the Tur\'an graph contains the most copies of $K_r$. Later, Gy\H ori, Pach, and Simonovits~\cite{GyoPacSim91} observed
that the same symmetrization argument can be used to count copies of any complete $r$-partite graph $H$ instead of $K_r$ and establish that the extremal graph is complete $(k-1)$-partite. However, in these cases, the extremal graph is not necessarily a Tur\'an graph. For a given $H$, it is a straightforward optimization problem to determine the complete $(k-1)$-partite graph that contains the maximum number of copies of $H$, but we cannot handle it in its full generality. Such optimization efforts for certain values of the parameters (some in different contexts) were carried out in~\cite{BroSid94,CutNirRad19,Ger20,GyoPacSim91,LiShi13,MaQiu18}.

\smallskip

Perhaps the most famous generalized Tur\'an problem is a conjecture of Erd\H os that asserts $\ex(5n,C_5,C_3)=n^5$. This conjecture is stated in several of his papers,
and we are not sure that we caught every appearance. The first reference we are aware of is~\cite{Erd75} from 1975, but there he already calls it an old conjecture. It was repeated in~\cite{Erd84,Erd86,Erd92}, sometimes as the more general claim $\ex((2r+1)n,C_{2r+1},\{C_3,C_5,\dots, C_{2r-1}\})=n^{2r+1}$.

Gy\H ori~\cite{Gyo89} showed $\ex(n,C_5,K_3)\le 1.03(\frac{n+1}{5})^5$, i.e., an upper bound with a somewhat larger constant factor than the lower bound, which is given by a blowup of the $C_5$. This was improved by F\"uredi to $\ex(n,C_5,K_3)\le 1.001(\frac{n+1}{5})^5$ in an unpublished manuscript, according to~\cite{HatHlaKra13}. The solution of the conjecture was obtained independently by Grzesik~\cite{Grz12} and by
Hatami, Hladk{\'y}, Kr{\'a}l', Norine, and Razborov~\cite{HatHlaKra13}. Both use the very powerful method of flag algebras (see Subsection~\ref{flag}), but with different approaches. Note that the obtained bound of $(n/5)^5$ is tight only when $n$ is divisible by 5. Neither of these papers solves the problem for every $n$, but~\cite{HatHlaKra13} shows that if $n$ is large enough, then an almost balanced blowup of $C_5$ has the maximum number of copies of $C_5$ among triangle-free graphs on $n$ vertices. Michael~\cite{Mic13} showed that for $n=8$, there is another extremal example in addition to the balanced blowup.
Lidick{\'y} and Pfender~\cite{LidPfe18} finished the solution of Erd\H os' conjecture by determining $\ex(n,C_5,C_3)$ for every $n$. The master's thesis of Siy~\cite{siy18} from 2018 collects the history of this problem and describes the various attacks on it. A stability version can be found in~\cite{PikSliTyr19}.

Grzesik and Kielak~\cite{GrzKie18} proved the general form of the conjecture 
\[
\ex((2r+1)n,C_{2r+1},\{C_3,C_5,\dots, C_{2r-1}\})=n^{2r+1}.
\] 
Somewhat surprisingly, for $r>2$, flag algebras were not needed in the proof. Note that they also showed that for every $n$, an almost balanced blowup of $C_{2r+1}$ is the only extremal graph. Beke and Janzer~\cite{BekJan23} strengthened this result by showing that the same holds for $\ex(n,C_{2r+1},C_{2r-1})$ if $n$ is sufficiently large. They also showed that for any $k$, if $r$ is sufficiently large, then an unbalanced blowup of $C_{2k-1}$ contains more copies of $C_{2r+1}$ than an almost balanced blowup; thus the latter is not an extremal graph for $\ex(n, C_{2r+1},C_{2k-1})$, disproving a conjecture of Grzesik and Kielak~\cite{GrzKie18}.

Bollob\'as and Gy\H ori~\cite{BolGyo08} considered the opposite problem, in a paper aptly titled ``Pentagons vs. triangles,'' and proved 
\[
(1+o(1))\frac{1}{3 \sqrt{3}}n^{3/2}\le \ex(n,K_3,C_5)\le (1+o(1))\frac{5}{4}n^{3/2}.
\] 
These bounds were later improved in~\cite{AloShi16,ErgGyoMet19b,ErgMet18,LvHeLu24}. Gy\H ori and Li~\cite{GyoLi12} extended this investigation to other forbidden odd cycles and proved the upper bound $\ex(n,K_3,C_{2k+1})=O(\ex(n,C_{2k}))$. The constant factor was improved in~\cite{FurOzk17,AloShi16}. Gy\H ori and Li also showed a lower bound of the form $\Omega(\ex(n,\{C_4,C_6,\dots,C_{2k}\}))$. It is an open problem for the ordinary Tur\'an number whether these lower and upper bounds have the same order of magnitude. The case where both graphs are cycles has received the most attention since then (see Section~\ref{cycles}).

\smallskip

The first paper to address when $H$ and $F$ are not selected from the most basic graph classes is by Gy\H ori, Pach, and Simonovits~\cite{GyoPacSim91}. They examined the number of copies of a general graph $H$ in $K_k$-free graphs. They observed $\ex(n,H,K_k)=\Theta(n^{|V(H)|})$ for any $H$ not containing $K_k$. Let us call (following~\cite{GerPal20}) a graph $H$ \emph{$k$-Tur\'an-good} if $0\neq\ex(n,H,K_k)$ is achieved by the Tur\'an graph $T(n,k-1)$, i.e., the Tur\'an graph has the maximum number of copies of $H$, provided $n$ is large enough. Unlike in (\ref{zykov-bound}), it may not be straightforward to count the number of copies of $H$ in $T(n,k-1)$. However, we typically consider the problem solved when we have determined an extremal graph.
In this language, Tur\'an's theorem states that $K_2$ is $k$-Tur\'an-good and Zykov's theorem (Theorem~\ref{zykov}) states that $K_r$ is $k$-Tur\'an-good. Gy\H ori, Pach, and Simonovits~\cite{GyoPacSim91} showed that every bipartite graph $H$ that contains a matching of size $\lfloor |V(H)|/2\rfloor$ is $3$-Tur\'an good. They also presented a large class of $(k-1)$-partite $k$-Tur\'an-good graphs and showed that $C_4$ and $K_{2,3}$ are $k$-Tur\'an-good for all $k$.   
In their theorems the threshold on $n$ is optimal, and they often proved that the Tur\'an graph is the unique extremal graph.

\smallskip

Fiorini and Lazebnik seem to be the first to explicitly frame the generalized Tur\'an question (see~\cite{FioLaz98,FioLaz94,Fio93}). 
However, they study particular cases in bipartite host graphs: $\ex_{\textrm{bip}}(n,C_6,C_4)$ and $\ex_{\textrm{bip}}(n,C_8,C_4)$. Their motivation comes from the observation that the incidence point-line graph of a finite projective plane, a well-known $C_4$-free graph, contains many copies of $C_6$ and $C_8$. De Winter, Lazebnik, and Verstra{\"e}te~\cite{DeWLazVer08} obtained similar results using partial planes.

\smallskip

In some cases, a result about counting subgraphs is just an intermediate result towards an ordinary Tur\'an-type theorem. We have shown one example already: the result of Erd\H os from 1938 where an upper bound on the number of cherries helped bound $\ex(n,C_4)$. From another point of view, this implies that a graph with many edges will contain many cherries---a simple supersaturation result. See~\cite{Sim84} for a more detailed description of how to apply theorems on supersaturated graphs with generalized Tur\'an results to prove extremal graph
theorems. 
While our goal in this survey is to find all related papers, such intermediate results make it hopeless, as the connection to this topic is often hidden inside proofs of other results.

 Similar to the argument for the number of cherries in $C_4$-free graphs, K\H ov\'ari, S\'os, and Tur\'an~\cite{KovSosTur54} in 1954 gave a bound on $\ex(n,K_{s,t})$ (Theorem~\ref{kst}) by observing first that in a $K_{s,t}$-free graph, every set of $s$ vertices has at most $t-1$ common neighbors, thus the number of \emph{stars} $S_{s+1}$ with $s$ leaves is at most $(t-1)\binom{n}{s}$. 
Other proofs using similar intermediate results on $\ex(n,H,F)$ appear in~\cite{ErdSim69,ErdSim82,ErdSim84,FurWes01,GrzJanNag19}.

\smallskip

We conclude this history arriving at the 2016 manuscript by Alon and Shikelman~\cite{AloShi16}. This article completely changed the landscape of the area, and the resulting boom of activity makes it impractical to organize results according to their chronology. We've mentioned a few directions of research that we feel have attracted the most attention historically. In the remainder of this survey, we attempt to give an exhaustive review of the literature; most of which has occurred since~\cite{AloShi16}.

\subsection{Counting edges---classical Tur\'an numbers}\label{edge-section}

In this subsection, we survey classical results for counting edges, i.e., the ordinary Tur\'an number $\ex(n,F)$. We do not attempt to be exhaustive, but rather highlight several key results that are often the target of generalizations in the later sections.
A reader familiar with the basic results in extremal graph theory can skip this section.
For an exhaustive overview of the topic see the monograph by Bollob\'as~\cite{Bol04}.

The starting point of extremal graph theory is Tur\'an's theorem, which determines the maximum number of edges in a graph without a $k$-clique subgraph.
The special case $k=3$ was proved by Mantel~\cite{Man07},
and is the earliest result in this area.

\begin{thm}[Tur\'an's Theorem~\cite{Tur41}]\label{turan}
	The Tur\'an graph $T(n,k-1)$ is the unique extremal graph for the complete graph $K_k$. Thus,
    \[
	\ex(n,K_k) = e(T(n,k-1)) \leq \left(1-\frac{1}{k-1}\right) \frac{n^2}{2}.
	\]
\end{thm}

Tur\'an's theorem was extended to any graph $F$ with a \emph{color-critical edge} (i.e.\ an edge whose removal decreases the chromatic number) by Simonovits.

\begin{thm}[Simonovits Critical Edge Theorem~\cite{Sim68}]\label{simo} 
	Let $F$ be a $k$-chromatic graph. For $n$ large enough, the 
    unique extremal graph for $F$ is $T(n,k-1)$
	if and only if $F$ has a color-critical edge.
\end{thm}

The following fundamental theorem of extremal graph theory extends Tur\'an's theorem.

\begin{thm}[Erd\H os and Stone~\cite{ErdSto46}; Erd\H os and Simonovits~\cite{ErdSim65}]\label{ESS}
	If $F$ is a graph with chromatic number $\chi(F)$, then
	$$\ex(n,F) = \left(1 - \frac{1}{\chi(F) -1}\right)\frac{n^2}{2} + o(n^2).$$
\end{thm}

A consequence of this theorem is the asymptotic value of $\ex(n,F)$ for all non-bipartite $F$. Thus, in such cases, we are generally now only interested in exact results. For example, 
Erd\H os, F\"uredi, Gould, and Gunderson~\cite{ErdFurGou95} determined the extremal number of the \emph{friendship graph} or \emph{$k$-fan} (i.e., $k$ triangles all sharing exactly one common vertex) $F_k$  for $n$ sufficiently large:
\begin{displaymath}
\ex(n,F_k)=\left\lfloor \frac{n^2}{4}\right \rfloor+
\left\{ \begin{array}{l l}
k^2-k & \textrm{if\/ $k$ is odd},\\
k^2-\frac{3k}{2} & \textrm{if\/ $k$ is even}.\\
\end{array}
\right.
\end{displaymath}

Let us turn to the case where $F$ is bipartite, the so-called \emph{degenerate case}. For an exhaustive survey, see F\"uredi and Simonovits~\cite{FurSim13}.
For complete bipartite graphs, even the order of magnitude is unknown in general, but the following upper bound is widely believed to be sharp.

\begin{thm}[Erd\H os; K\H ov\'ari-S\'os-Tur\'an Theorem~\cite{KovSosTur54}]\label{kst}
	Let $s \leq t$. Then
	\[
	\ex(n,K_{s,t}) \leq \frac{1}{2}(t-1)^{1/s} n^{2-{1/s}}+O(n).
	\]
\end{thm}

A series of papers~\cite{KolRonSza96, AloRonSza99, Buk21}  give general lower bounds for $t$ large compared to $s$.
The case $s=2$ is especially important as a construction of F\"uredi~\cite{Fur96} determines the asymptotics.

\begin{thm}[F\"uredi~\cite{Fur96}]\label{fured}
	\[
	\ex(n,K_{2,t}) = \frac{1}{2} \sqrt{t-1}n^{3/2} + O(n).
	\]
\end{thm}




For even cycles, the situation is similar to complete bipartite graphs: we have a simple upper bound that is widely believed to have the correct order of magnitude, but it is matched by a lower bound in only a few cases.

\begin{thm}[Bondy-Simonovits Even Cycle Theorem~\cite{BonSim74}]\label{BS-even-cycle}
	\[
	\ex(n,C_{2k}) \leq 100kn^{1+1/k}.
	\]
\end{thm}

The bound $\ex(n,C_{2k})=O(n^{1+1/k})$ was stated by Erd\H os in~\cite{Erd64} without a proof.
The constant factor has been improved since and the best known is $16\sqrt{5k \log k} + o(1)$ due to He~\cite{He21}.
Erd\H os~\cite{Erd64} conjectured that we have the same order of magnitude if we forbid all cycles of length at most $2k$.

\begin{conjecture}[Erd\H os Girth Conjecture]\label{erdgirth} 
For any positive integer $k$, there is a constant $c(k)>0$ such that there exists an $n$-vertex graph with at least $c(k)n^{1+1/k}$ edges and girth greater than $2k$. In other words,
\[
\ex(n,\cC_{\leq 2k}) = \Omega(n^{1+1/k})
\]
where $\cC_{\leq 2k}$ denotes the family of all cycles of length at most $2k$.
\end{conjecture}

Let us continue with trees $T$. It is easy to see that $\ex(n,T)=\Theta(n)$ if $T$ has at least two edges. The case $T$ is a path is given by Erd\H os and Gallai~\cite{ErdGal59}.

\begin{thm}[Erd\H os-Gallai Theorem~\cite{ErdGal59}]
For the $k$-vertex path $P_k$, 
	$$\ex(n,P_k) \leq \frac{n}{k-1}\binom{k-1}{2} = \frac{k-2}{2}n.$$ 
\end{thm}




The Erd\H os-Gallai bound is conjectured to hold for all trees.

\begin{conjecture}[Erd\H os-S\'os Conjecture~\cite{Erd64}]\label{erdsos}
	Let $T$ be a $k$-vertex tree. Then
	\[
	\ex(n,T) \leq \frac{k-2}{2}n.
	\]
\end{conjecture}

If true, this is sharp when $k-1$ divides $n$, as shown by $\frac{n}{k-1}K_{k-1}$, the graph of $\frac{n}{k-1}$ pairwise vertex-disjoint copies of $K_{k-1}$.

It is easy to see that the conjecture holds for stars. It has also been proved for other classes of trees beyond paths.

There are several results in this area that go beyond just determining the Tur\'an number. We mention two of the most fundamental. The first is {\it stability}---the notion that an $F$-free graph with ``almost'' the extremal number of edges will have ``almost'' the same structure as the Tur\'an graph.

\begin{thm}[Erd\H os-Simonovits Stability Theorem~\cite{Sim68,Erd66,Erd68}]\label{esstab} For $\varepsilon>0$ and any graph $F$ with $\chi(F)\ge 3$, there exists $\delta>0$ such that if $G$ is an $n$-vertex $F$-free graph with
\[
e(G) >\ex(n,F)-\delta n^2,
\]
then the edit distance between $G$ and $T(n,\chi(F)-1)$ is at most $\varepsilon n^2$.
In other words, we can add and delete at most $\varepsilon n^2$ edges of $G$ to obtain $T(n,\chi(F)-1)$.
\end{thm}

The second is {\it supersaturation}---if a graph has $\varepsilon n^2$ more edges than the extremal number for $F$, then it will contain not just one copy of $F$, but many.

\begin{thm}[Erd\H os-Simonovits Supersaturation Theorem~\cite{ErdSim84}]\label{essuper}
	For $\varepsilon>0$ and any graph $F$, there exists $\delta >0$ such that if $G$ is an $n$-vertex graph with
	\[
	e(G) > \ex(n,F) + \varepsilon n^2,
	\]
	then $G$ contains $\delta n^{|V(F)|}$ copies of $F$.
\end{thm}

\section{General results}\label{gener}

Most statements about $\ex(n,H,F)$ involve $H$ and $F$ from two specific simple graph classes. In this section, we gather results where at least one of $H$ and $F$ is somewhat general.

For ordinary Tur\'an numbers, the Erd\H os-Stone-Simonovits Theorem (Theorem~\ref{ESS}) gives asymptotics on $\ex(n,F)$ for all graphs $F$ of chromatic number at least $3$. This leaves the extremal numbers for bipartite graphs wide open; this is the \textit{degenerate problem}. The analogous result in the generalized Tur\'an setting was proved by Alon and Shikhelman~\cite{AloShi16}.

\begin{thm}[Alon and Shikhelman~\cite{AloShi16}]\label{ESS-AS} Let $F$ be a graph with $\chi(F)=k$. Then 
    \[
	\ex(n,K_r,F)=\binom{k-1}{r}\left(\frac{n}{k-1}\right)^r+o(n^r).
	\]
\end{thm}

Note that if $r\ge k$, then $\binom{k-1}{r}=0$ and so the above bound reduces to $o(n^r)$.
Thus, Theorem~\ref{ESS-AS} extends the Erd\H os-Stone-Simonovits Theorem to count copies of $K_r$, and now every graph with chromatic number at most $r$ is degenerate.
Gerbner and Palmer~\cite{GerPal18} gave a further extension to count arbitrary graphs $H$.

\begin{restatable}[Gerbner and Palmer~\cite{GerPal18}]{thm}{generalboundthm}\label{genESS} 
Let $F$ be a graph with $\chi(F)=k$. Then 
\[
\ex(n,H,F) = \ex(n,H,K_k)+o(n^{|V(H)|}).
\]
\end{restatable}

We will give a sketch of the proof of this theorem in Subsection~\ref{regularity}, to illustrate how the regularity lemma can be applied to the generalized Tur\'an problem.

An important description of the degenerate case for general $H$ and $F$ is given by Alon and Shikhelman~\cite{AloShi16}. The \emph{blowup} $G[t]$ of a graph $G$ is the graph resulting from replacing each vertex of $G$ with $t$ copies of itself.

\begin{prop}[Alon and Shikhelman~\cite{AloShi16}]\label{degen-prop}
We have $\ex(n,H,F) = o(n^{|V(H)|})$ if and only if $F$ is a subgraph of a blowup of $H$. Otherwise, $\ex(n,H,F) = \Omega(n^{|V(H)|})$.
\end{prop}

So we now say that the \emph{degenerate case} is when $\ex(n,H,F) = o(n^{|V(H)|})$ and the \emph{non-degenerate case} is when $\ex(n,H,F) = \Omega(n^{|V(H)|})$.

\smallskip

Gerbner and Palmer~\cite{GerPal18} established a very basic (but sometimes practical) bound of
\[
\ex(n,H,F)\geq \ex(n,F)-\ex(n,H).
\]
The argument is simple: removing an edge from each copy of $H$ in an $F$-free graph yields an $H$-free graph.
They also proved the following general lower bound:
\begin{prop}[Gerbner and Palmer~\cite{GerPal18}]\label{gen-LB}
If $e(F)|>e(H)$, then
\[
   \ex(n,H,F) = \Omega\left(n^{|V(H)| - \frac{e(H)(|V(F)|-2)}{e(F)-e(H)}}\right).
\]  
\end{prop}

We will present the proofs of Proposition~\ref{degen-prop} and \ref{gen-LB} in Subsection~\ref{probab}, to illustrate how probabilistic methods can be applied in the generalized setting.

Gishboliner and Shapira~\cite{GisSha18} gave the following general lower bound, where $\alpha(G)$ denotes the independence number of $G$.

\begin{proposition}[Gishboliner and Shapira~\cite{GisSha18}]\label{gissha}
Let $|V(F)|-\alpha(F) > |V(H)|-\alpha(H)$.
Then
\[
 \ex(n,H,F) = \Omega_{|V(H)|}\left((|V(F)|-\alpha(F))^{|V(H)|-\alpha(H)}n^{\alpha(H)}\right) .
\]
\end{proposition}

We sketch their proof: blow up each vertex of a largest independent set of $H$ to linear size, then there are $\Omega(n^{\alpha(H)})$ copies of $H$. The resulting graph is still $F$-free, as each edge is incident to at least one of the $|V(H)|-\alpha(H)$ vertices that are not blown up, but the smallest such set (i.e.\ a vertex cover) in $F$ has $|V(F)|-\alpha(F) > |V(H)|-\alpha(H)$ vertices. 

The Simonovits Critical Edge Theorem (Theorem~\ref{simo}) has the following extension.

\begin{thm}[Ma and Qiu~\cite{MaQiu18}]\label{gen-crit-edge}
	Let $F$ be a $k$-chromatic graph. For $r<k$ and $n$ large enough, the unique $n$-vertex $F$-free graph with the
    maximum number of copies of $K_r$ is $T(n,k-1)$
	if and only if $F$ has a color-critical edge.
\end{thm}

The classical stability (Theorem~\ref{esstab}) and supersaturation (Theorem~\ref{essuper}) theorems have analogues in the generalized setting. The general stability theorem is:

\begin{theorem}[Ma and Qiu~\cite{MaQiu18}]\label{maqiustabi} 
 For $\varepsilon>0$ and any graph $F$ with $\chi(F) > r \geq 3$, there exists $\delta>0$ such that if $G$ is an $n$-vertex $F$-free graph with at least $\ex(n,K_r,F)-\delta n^r$ copies of $K_r$, then the edit distance between $G$ and $T(n,\chi(F)-1)$ is at most $\varepsilon n^2$.
\end{theorem}

Additional proofs of this theorem were obtained by Liu~\cite{Liu19},  Halfpap and Palmer~\cite{HalPal19}, and Li and Pang~\cite{LiPen22}.

The general supersaturation theorem is as follows.

\begin{theorem}[Halfpap and Palmer~\cite{HalPal19}]
	For $\varepsilon>0$ and any graph $F$, there exists $\delta >0$ such 
 that any $n$-vertex graph with more than $\ex(n,H,F)+\varepsilon n^{|V(H)|}$ copies of $H$ contains $\delta n^{|V(F)|}$ copies of $F$.
\end{theorem}

Note that the statement in \cite{HalPal19} includes the assumption that $\chi(H)<\chi(F)$, but this can be safely dropped without impacting the conclusion.

\subsection{Differences between generalized and classical Tur\'an problems}\label{diff}

Both extremal functions $\ex(n,F)$ and $\ex(n,H,F)$ denote the maximum number of copies of a particular subgraph $H$ and so $\ex(n,H,F)$ is a direct generalization of $\ex(n,F)$. The more general function often behaves analogously to the classical function, but there are important differences. We highlight several in this subsection.

The value $\ex(n,F)$ immediately determines every possible value of $e(G)$ for an $F$-free graph $G$. Indeed, by removing the edges one by one, we can obtain $F$-free graphs with $0,1,2,\dots, \ex(n,F)$ edges. However, there are instances where there is no $n$-vertex $F$-free graph with exactly $\ex(n,H,F)-1$ copies of $H$ as we cannot always delete a single copy of $H$. We are not aware of any result concerning the range of values attainable by $\cN(H,G)$ when $G$ is an $n$-vertex $F$-free graph.

\smallskip

The \emph{rational exponents conjecture} of Erd\H os and Simonovits~\cite{Erd66,Erd75,ErdSim82,Erd88} claims that for any graph $F$, there is a rational $1\le \alpha \le 2$ and a constant $c>0$ such that $\ex(n,F)=(c+o(1))n^\alpha$. However, an analogous claim cannot hold when counting triangles instead of edges: Alon and Shikhelman~\cite{AloShi16} showed that $n^{2-o(1)}\le \ex(n,K_3,B_k)\le o(n^2)$ for every $k$ (where $B_k$ is the \emph{book} with $k$ \emph{pages}, i.e., $k$ triangles all sharing a common edge).
 This result immediately follows from a result of Clark, Entringer, McCanna, and Sz\'ekely~\cite{ClaEntMcC91}, which establishes the same bounds for the number of edges in a graph where every edge is in exactly one triangle. This itself is a consequence of the celebrated (6,3)-problem of Ruzsa and Szemer\'edi~\cite{RuzSze78}.

Let us note here that the other direction is an even more famous conjecture of Erd\H os and Simonovits~\cite{Erd75,Erd88}: for every rational $1\le \alpha \le 2$ there is a graph $F$ such that $\ex(n,F)=(c+o(1))n^\alpha$ for some constant $c$ depending only on $F$. There are other versions of this conjecture, where $F$ is replaced by a family of graphs and/or the asymptotic statement is replaced by a statement on the order of magnitude: $\ex(n,F)=\Theta(n^\alpha)$. 
Perhaps it would be easier to prove the following corresponding generalized version.

\begin{restatable}{conjecture}{rationalconj}\label{rationa}
	For every rational $\alpha\ge 1$, there are graphs $H$ and $F$ such that
 \[
 \ex(n,H,F)=\Theta(n^\alpha).
 \]
\end{restatable}

A step towards this conjecture is due to English, Halfpap, and Krueger~\cite{EngHalKru24}, who showed that for every $r$ and every rational $\alpha$ with $1\le \alpha\le r$ there is a family of forbidden graphs $\cF_{\alpha,r}$ such that $\ex(n,K_r,\cF_{\alpha,r})=\Theta(n^\alpha)$.

\smallskip

Another famous conjecture of Erd\H os and Simonovits~\cite{ErdSim82} is the \emph{compactness conjecture} that states that for any finite family $\cF$ of graphs there is an $F \in \cF$ such that $\ex(n,\cF)=O(\ex(n,F))$. In other words, if we are only interested in the order of magnitude of the Tur\'an function, it is sufficient to forbid a single graph from $\cF$. Note that this conjecture does not necessarily hold for infinite families. For example, forbidding the family of all cycles forces at most $n-1$ edges in an $n$-vertex graph, but $\ex(n,C_k)$ is superlinear for all $k$. When counting graphs other than $K_2$, a corresponding compactness conjecture does not need to hold. Indeed, it is not hard to show that $\ex(n,C_6,C_4)=\Theta(n^3)$ and $\ex(n,C_6,C_8)=\Theta(n^3)$, but $\ex(n,C_6,\{C_4,C_8\})=\Theta(n^2)$ (see Subsection~\ref{cycles} for more details). It would be interesting to examine whether there is a graph $H$ such that any finite family $\cF$ of graphs must contain a graph $F \in \cF$ satisfying $\ex(n,H,\cF)=O(\ex(n,H,F))$.

\smallskip

It is well known that $\ex(n,tF)=\Theta(\ex(n,F))$, where $tF$ denotes the graph of $t$ pairwise vertex-disjoint copies of $F$. Moreover, Gorgol~\cite{Gor11} observed that $\ex(n,tF)=\ex(n,F)+O(n)$ and so even the asymptotics of $\ex(n,tF)$ and $\ex(n,F)$ are the same as long as $F$ is not a forest. On the other hand, Alon and Shikhelman~\cite{AloShi16} observed that $\ex(n,K_3,2C_5)=\Omega(n^2)$ in contrast with $\ex(n,K_3,C_5)=O(n^{3/2})$ given in~\cite{BolGyo08}. The construction establishing the quadratic lower bound is very simple: take a $C_5$-free graph on $n-1$ vertices with a quadratic number of edges and add a \emph{universal} vertex $v$ (i.e., $v$ is adjacent to all other vertices). Each copy of $C_5$ in this graph must contain $v$ and so there is no $2C_5$. As $v$ is universal, each edge in the neighborhood of $v$ is a triangle in the graph.
Therefore, $\ex(n,K_3,2C_5)\ge \ex(n-1,K_2,C_5) = \Omega(n^2)$.
In fact, $\ex(n-1,K_2,C_5)=\lfloor (n-1)^2/4\rfloor$ for sufficiently large $n$ by Theorem~\ref{simo}, and $\ex(n,K_3,2C_5)=\lfloor (n-1)^2/4\rfloor$ for sufficiently large $n$ (see~\cite{ZhaCheGyo23}).

Universal vertices can be further exploited.
 More generally, if $H'$ is obtained from $H$ by removing at most $t-1$ vertices, then by adding $t-1$ universal vertices to an extremal graph for $\ex(n-t+1,H',F)$ we obtain $\ex(n,H,tF)\ge \ex(n-t+1,H',F)$. 

Gerbner, Methuku, and Vizer~\cite{GerMetViz17} studied the phenomenon where $\ex(n,H,tF)$ and $\ex(n,H,F)$ diverge for several different graphs $H$ and $F$. They showed that the above reason is the only one: $\ex(n,H,F)=o(\ex(n,H,tF))$ is only possible if $H$ has an induced subgraph $H'$ such that $\ex(n,H,F)=o(\ex(n,H',F))$.

\smallskip

It is easy to see that $\ex(n,F)\ge \lfloor n/2\rfloor$ for every $F\neq K_2$. Indeed, when $F$ is not a matching, construct a matching and when $F$ is a matching, construct a star. For generalized Tur\'an problems, other nonzero sublinear functions are possible. A very simple example is $\ex(n,K_3,2K_2)=1$ as a $2K_2$-free graph with a triangle cannot contain any additional edges. Gerbner and Methuku~\cite{GerMet25} examined this phenomenon more thoroughly and characterized the graphs $H$ and $F$ that satisfy $\ex(n,H,F)=O(1)$. They also showed that there is nothing between constant and linear: if $\ex(n,F)=\omega(1)$, then $\liminf_{n \to \infty}  \frac{\ex(n,H,F)}{n} \ge \frac{1}{|V(H)|}$. Finally, for every nonnegative integer $m$, they constructed graphs $H$ and $F$ such that $\ex(n,H,F)=m$, for $n$ sufficiently large.

\section{The non-degenerate case}\label{turangood}

The non-degenerate case is where $H$ and $F$ satisfy $\ex(n,H,F)=\Theta(n^{|V(H)|})$. Proposition~\ref{degen-prop} characterizes when this happens: if and only if $F$ is not a subgraph of a blowup of $H$. 
A sufficient condition is $\chi(H)<\chi(F)$, but it is not necessary. As an example, observe that any blowup of any triangle-free graph $H$ is also triangle-free, hence $\ex(n,H,K_3)=\Theta(n^{|V(H)|})$, even if $\chi(H)=3$. It is well known that there are triangle-free graphs with arbitrary chromatic number, thus $\chi(H)$ can be much larger than $\chi(F)$ even in the non-degenerate case. On the other hand, if $F$ is bipartite and $H$ has at least one edge, then we are in the degenerate case. 

In Subsection~\ref{history} we stated bounds on $\ex(n,C_{2r+1},C_{2k+1})$ for $k <r$. Surprisingly, we are not aware of any other non-degenerate result with $\chi(H)\ge \chi(F)$. Thus, in the rest of this section, we restrict our attention to the case when $\chi(H)<\chi(F)$.

As in Subsection~\ref{history}, a graph $H$ is $k$-Tur\'an-good if $\ex(n,H,K_k)=\cN(H,T(n,k-1))$ for $n$ large enough. We introduce the following generalizations. Given graphs $H$ and $F$ with $\chi(F)=k$, we say that $H$ is \textit{$F$-Tur\'an-good} if $\ex(n,H,F)=\cN(H,T(n,k-1))$ for $n$ large enough. We say that $H$ is \textit{weakly $F$-Tur\'an-good} if $\ex(n,H,F)=\cN(H,T)$ for some complete $(k-1)$-partite graph $T$. When $F=K_{k}$, we use the briefer term \textit{(weakly) $k$-Tur\'an-good} instead of (weakly) $K_{k}$-Tur\'an-good. 
With these definitions, Tur\'an's Theorem (Theorem~\ref{turan}) states that $K_2$ is $k$-Tur\'an-good for all $k$, the Simonovits Critical Edge Theorem (Theorem~\ref{simo}) states that $K_2$ is $F$-Tur\'an-good if and only $F$ has a color-critical edge, and Zykov's Theorem (Theorem~\ref{zykov}) states that if $r < k$, then $K_r$ is $k$-Tur\'an-good. Theorem~\ref{gen-crit-edge} states that $K_r$ is $F$-Tur\'an-good if and only if $F$ has a color-critical edge.

It is not hard to see that the weak version of Tur\'an-goodness is necessary as, for example, the complete bipartite graph with the maximum number of copies of the star $S_{r+1} = K_{1,r}$ is not a Tur\'an graph when $r$ is large. That is, $S_{r+1}$ is weakly $k$-Tur\'an-good, but not $k$-Tur\'an-good.
In general, when $H$ is unbalanced, we can expect that an unbalanced complete multipartite graph maximizes the number of copies of $H$. Determining which complete multipartite graph contains the maximum yields an optimization problem that tends to be difficult to solve in this general setting. Some papers that investigate the extremal graph in the special case where $H$ itself is a complete multipartite graph (typically bipartite) are referenced in Subsection~\ref{history}. We are not aware of any such optimization for other graphs $H$ (outside the obvious case where the extremal graph is a Tur\'an graph).

Let us mention a connection to another well-studied concept in extremal graph theory: \textit{inducibility}, which is the maximum number of induced copies of a graph $H$ in an $n$-vertex graph. 
Here we highlight some simple examples.
Like the result of Gy\H ori, Pach, and Simonovits stated in Subsection~\ref{history}, Zykov's symmetrization can be used to show that the maximum number of induced copies of a complete $r$-partite graph occurs in a complete multipartite graph $G$. 
However, $G$ is not necessarily $r$-partite. This was observed by Brown and Sidorenko~\cite{BroSid94}, who also showed that if $r=2$, then $G$ must be bipartite. Gerbner~\cite{Ger21b} showed that $K_{a,b}$ is weakly $F$-Tur\'an-good for all values of $a,b$ and $F$ a $3$-chromatic graph with a color-critical edge. If we allow $a,b$ to both be large enough, then the argument is simple: adding any edge to $K_{a,b}$ results in a copy of $F$ and so every copy of $K_{a,b}$ in an $F$-free graph must be induced, which implies that the extremal graph is complete bipartite.

\subsection{Tur\'an-stable graphs}\label{turstab}

We also introduce approximate versions of the (weakly) Tur\'an-good property.
Given $F$ with $\chi(F)=k>2$, we say that $H$ is \textit{asymptotically $F$-Tur\'an-good} if $\ex(n,H,F)=(1+o(1))\cN(H,T(n,k-1))$ and $H$ is \textit{asymptotically weakly $F$-Tur\'an-good} if $\ex(n,H,F)=(1+o(1))\cN(H,T)$ for some complete $(k-1)$-partite graph $T$. Theorem~\ref{genESS} implies that if $H$ is asymptotically $K_{k}$-Tur\'an-good, then $H$ is also asymptotically $F$-Tur\'an-good for every graph $F$ with chromatic number $k$.

Given $F$ with $\chi(F)=k>2$, we say that $H$ is \textit{$F$-Tur\'an-stable} if any $n$-vertex $F$-free graph 
with $\ex(n,H,F)-o(n^{|V(H)|})$ copies of $H$ 
has edit distance $o(n^2)$ from the $(k-1)$-partite Tur\'an graph $T(n,k-1)$. We say that $H$ 
is \emph{weakly $F$-Tur\'an-stable} if any $n$-vertex $F$-free graph 
with $\ex(n,H,F)-o(n^{|V(H)|})$ copies of $H$ has edit distance $o(n^2)$ from some complete $(k-1)$-partite
graph. As usual, for brevity, we use the term \emph{(weakly) $k$-Tur\'an-stable} in place of (weakly) $K_k$-Tur\'an-stable.

Clearly, the $F$-Tur\'an-stable property implies the asymptotically $F$-Tur\'an-good property, and the weakly $F$-Tur\'an-stable property implies the asymptotically weakly $F$-Tur\'an-good property. But these properties can also be used to derive  exact results.
Two important facts concerning Tur\'an stability are: 
\begin{itemize}
    \item If $H$ is (weakly) $k$-Tur\'an-stable, then $H$ is (weakly) $F$-Tur\'an-stable for every $F$ with chromatic number $k$.

    \item If $H$ is weakly $F$-Tur\'an-stable and $F$ has a color-critical edge, then $H$ is weakly $F$-Tur\'an-good.
\end{itemize}

Note that the second fact holds only for the weak properties. Indeed, if $H$ is $F$-Tur\'an-stable, then $H$ is weakly $F$-Tur\'an-stable, thus weakly $F$-Tur\'an-good, but the extremal graph is not necessarily the Tur\'an graph (although it does have edit distance $o(n^2)$ from the Tur\'an graph). An example is $\ex(n, S_4,K_3)$. As we shall see, $S_4$ is $K_3$-Tur\'an-stable, thus the extremal graph is $K_{m,n-m}$ for some $m$. Brown and Sidorenko~\cite{BroSid94} showed (in a different setting) that $m$ is roughly $n/2+ \sqrt{(3n-2)/2}$ and so $S_4$ is not $K_3$-Tur\'an-good. 
On the other hand, if $H$ is weakly $F$-Tur\'an-stable and $k$-Tur\'an-good, then $H$ is $F$-Tur\'an-good.

Among other things, the two facts mean that if we prove that $H$ is 
(weakly) $k$-Tur\'an-stable,
then we immediately arrive at infinitely many 
exact results. 
The simple proofs of the two bulleted facts and more details can be found in Subsection~\ref{stabmeth} where we discuss the \emph{Stability Method}.

An advantage of using Tur\'an-stability to prove an exact result is that we need less precision. For example, since $K_{k-1}$ is $k$-Tur\'an-good, it is immediate that $2K_{k-1}$ is also $k$-Tur\'an-good:  first select a copy of $K_{k-1}$, the number of which is maximized by the Tur\'an graph. Then select a second copy, which is maximized by the Tur\'an graph on $n-(k-1)$ vertices. Observe that this approach does not work for $2K_r$ with $r<k-1$, since after taking the first copy of $K_r$, the remaining graph may not be a Tur\'an graph. However, by a result of Ma and Qiu~\cite{MaQiu18}, we know that $K_r$ is $k$-Tur\'an-stable and so after selecting the first copy of $K_r$, the resulting graph has edit distance $O(1)$ from a Tur\'an graph. This implies that $2K_r$ is $k$-Tur\'an-stable, thus weakly $k$-Tur\'an-good. It is not difficult to show that the Tur\'an graph contains the maximum number of copies of $2K_r$ among the complete $(k-1)$-partite graphs and thus conclude that $2K_{k-1}$ is indeed $k$-Tur\'an-good. 
The drawback of this approach is that the first argument gives $\ex(n,2K_{k-1},K_{k})=\cN(K_{k-1},T(n,k-1))$ for every $n$, while the second gives $\ex(n,2K_r,K_{k})=\cN(K_r,T(n,k-1))$ only for an unspecified sufficiently large $n$.

Several results for $k$-Tur\'an-good graphs were later proved for $k$-Tur\'an-stable graphs. In fact, we are aware of only a few exact non-degenerate results on $\ex(n,H,F)$ for $F$ a non-clique with $\chi(F)=k$, that are not a consequence of $H$ being $k$-Tur\'an-stable. 

\subsection{$k$-Tur\'an-good/stable graphs}\label{good-stable}

Before examining specific graphs, let us discuss what is known for $k$-Tur\'an-good graphs in general. Which graphs are $k$-Tur\'an-good? An answer is that every graph $H$ is $k$-Tur\'an-good, provided $k\ge 300|V(H)|^9$, by Morrison, Nir, Norin, Rza{\.z}ewski, and Wesolek~\cite{MorNirNor22}, who established the first bound, but did not attempt to optimize the threshold. This result was extended by Gerbner and Hama Karim~\cite{GerHam23} to show that $H$ is also $k$-Tur\'an-stable for the same threshold on $k$. 

Gy\H ori, Pach, and Simonovits~\cite{GyoPacSim91} constructed a tree that is not asymptotically weakly 3-Tur\'an-good. 
For all $k$, Grzesik, Gy\H ori, Salia, and Tompkins~\cite{GrzGyoSal22} constructed a $(k-1)$-chromatic graph $H$ that is not asymptotically weakly $k$-Tur\'an-good. They conjectured that any $(k-1)$-chromatic graph with diameter at most $2(k-1)$ is $k$-Tur\'an-good, but this was disproved by Hng and Cecchelli~\cite{HngCec22} for $k\ge 4$. This conjecture remains open for $k=3$.
Gerbner~\cite{Ger22d} constructed, for every $F$, a graph that is not asymptotically weakly $F$-Tur\'an-good. However, we do not know of any example of a $(k-2)$-chromatic graph that is not weakly $k$-Tur\'an-good. 

On the other hand, Gerbner~\cite{Ger20c} showed that there exist $F$-Tur\'an-good (and weakly $F$-Tur\'an-good) graphs if and only if $F$ has a color-critical vertex. Now let $F$ be a graph without a color-critical vertex. 
For a complete $n$-vertex $(k-1)$-partite graph $T$, consider a graph obtained by selecting a vertex of $T$ and connecting it by an edge to all the other vertices. The resulting graph remains $F$-free, and contains more copies of $H$ than does $T$, for every fixed graph $H$ of chromatic number at most $k$ (and $n$ sufficiently large). This implies that $H$ is not $F$-Tur\'an-good. Note that if $F$ does not have a color-critical edge and $\chi(F)=k$, then the complete $(k-1)$-partite graphs are not maximal: adding an edge does not create a copy of $F$. However, it is also possible that this new edge fails to create a copy of $H$. The first such example is by Gerbner and Palmer~\cite{GerPal20}, who showed that $C_4$ is $F_2$-Tur\'an-good (where $F_2$ is the \emph{bowtie graph}, i.e., two triangles sharing exactly one vertex).

Gerbner and Hama Karim~\cite{GerHam23} showed that for every bipartite graph $H$ there is a $3$-chromatic graph $F$ such that $H$ is weakly $F$-Tur\'an-good (moreover, $H$ is weakly $F$-Tur\'an-stable). Even more is true: for every bipartite graph $H$, there is a $k_0$ such that $H$ is weakly $C_{2k+1}$-stable if $k\ge k_0$.

Interestingly, while the known examples of graphs $H$ that are not weakly $k$-Tur\'an-good each have chromatic number $k-1$, most constructions of $k$-Tur\'an-good graphs $H$ are also $(k-1)$-chromatic.
 Here we list most of the known exceptions. The above result of Morrison, Nir, Norin, Rza{\.z}ewski, and Wesolek~\cite{MorNirNor22} deals specifically with the case when $k$ is large compared to the chromatic number of $H$. As stated in Subsection~\ref{history}, Gy\H ori, Pach, and Simonovits~\cite{GyoPacSim91} showed that complete $r$-partite graphs are weakly $k$-Tur\'an-good, if $r< k$. This was extended to $k$-Tur\'an-stability by Gerbner and Hama Karim~\cite{GerHam23}. For the specific case $H=C_4$, $k$-Tur\'an-stability was proved earlier in~\cite{HeiHou22}.
The first Tur\'an-stability result (aside from the Erd\H os-Simonovits stability theorem) is due to
Ma and Qiu~\cite{MaQiu18} (Theorem~\ref{maqiustabi}), who showed that if $r<k$, then $K_r$ is $k$-Tur\'an-stable.  

Lidick{\'y} and Murphy~\cite{LidMur20} showed that $C_5$ is $k$-Tur\'an-stable for every $k\ge 4$. Hei, Hou, and Liu~\cite{HeiHouLiu21} proved that paths are weakly $k$-Tur\'an-stable for $k\ge 3$. Gerbner~\cite{Ger22c} showed that paths are in fact $k$-Tur\'an-stable and also $k$-Tur\'an-good, which was a conjecture of Gerbner and Palmer~\cite{GerPal20} with partial results in~\cite{Ger20,MurNir21,QiaXieGe21,HeiHouLiu21}. Moreover, \cite{Ger22c} includes the stronger result that all linear forests are $k$-Tur\'an-stable (a forest is \emph{linear} if its components are paths).
 This implies that the $r$-edge matching $M_r$ is $F$-Tur\'an-good for any graph $F$ with a color-critical edge, an earlier result of Gerbner~\cite{Ger20}. 
Qian, Xie, and Ge~\cite{QiaXieGe21} proved that if an $r$-vertex graph contains $K_{r-1}$, then it is $k$-Tur\'an-good for every $k$. 
Furthermore, if an $r$-vertex graph contains $K_{r-2}$, then it is $r$-Tur\'an-good.
They also proved that the vertex-disjoint union $K_r \cup K_{r'}$ is $k$-Tur\'an-good for all $r\leq r' < k$.
Gerbner~\cite{Ger24c} found that the double-star $S_{1,2}$ is $4$-Tur\'an-good.

Let us continue by presenting some general $(k-1)$-chromatic $k$-Tur\'an-good graphs.

\begin{thm}[Gy\H ori, Pach, and Simonovits~\cite{GyoPacSim91}]\label{gpl} 
Let $r\ge 3$ and let $H$ be a $(k-1)$-partite graph with $m>k-1$ vertices,
containing $\lfloor m/(k-1)\rfloor$ vertex-disjoint copies of $K_{k-1}$. Suppose
further that for any two vertices $u$ and $v$ in the same connected component of $H$, there is a sequence $A_1,\dots,A_s$ of $(k-1)$-cliques in $H$ such that $u\in A_1$, $v\in A_s$, and for all $i<s$, $A_i$ and $A_{i+1}$ share $k-2$ vertices.
Then $H$ is $k$-Tur\'an-good. Moreover, if $n$ is large enough, the Tur\'an graph is the unique $n$-vertex $K_k$-free graph with $\ex(n,H,K_k)$ copies of $H$.
\end{thm}

In particular, for $k=3$, this implies that bipartite graphs $H$ that have a matching of size $\lfloor (|V(H)|-1)/2\rfloor$ are $3$-Tur\'an-good. This was extended by Gorovoy, Grzesik, and Jaworska~\cite{GorGrzJaw23}, to establish that if a bipartite graph $H$ has a matching of size $r$ and at most $\sqrt{r-1}$ unmatched vertices, then $H$ is $3$-Tur\'an-good. They also showed that $\sqrt{r-1}$ cannot be replaced by any linear function.

Gerbner and Palmer~\cite{GerPal20} coined the terminology of $k$-Tur\'an-good graphs and obtained a theorem of a similar flavor to Theorem~\ref{gpl}.

\begin{thm}[Gerbner and Palmer~\cite{GerPal20}]\label{turgood}
	Fix a $k$-Tur\'an-good graph $H$. Let $H'$ be any graph constructed from $H$ as follows.
	Choose a complete subgraph of $H$ with vertex set $X$, add a vertex-disjoint copy of $K_{k-1}$ to $H$ and join the vertices in $X$ to the vertices of the new copy of $K_{k-1}$ by edges arbitrarily.
	Then $H'$ is $k$-Tur\'an-good.
\end{thm} 

Both of the above theorems use $K_{k-1}$ as a building block for constructing $k$-Tur\'an-good graphs and have a property restricting how these blocks ``connect'' to each other. The first theorem prescribes a very strong connection property, but edges can be added only if they do not increase the chromatic number. In the second theorem, the final copy of $K_{k-1}$ is connected to only one of the earlier copies, but arbitrarily.

Gerbner~\cite{Ger20c} proved a generalization of Theorem~\ref{gpl}, where we begin with a $k$-Tur\'an-good graph $H$ that has a unique $(k-1)$-coloring, then add copy of $K_{k-1}$ and new edges in such a way that satisfies the requirement in Theorem~\ref{gpl} on a sequence of $(k-1)$-cliques in $H$.
Gerbner~\cite{Ger22b} obtained further generalizations, including a version in the $k$-Tur\'an-stability regime.

Gerbner~\cite{Ger22d} showed that the double-star $S_{a,b}$ is weakly 3-Tur\'an-stable (implying a weakly 3-Tur\'an-good result from~\cite{GyoWanWoo21} and its extension to other 3-chromatic graphs $F$ with a color-critical edge in~\cite{Ger21e}). Let $H$ be a bipartite graph that contains a complete bipartite graph $K$ such that every vertex of $H$ is adjacent to a vertex of $K$. Gerbner~\cite{Ger22d} showed that $H$ is weakly 3-Tur\'an-stable, thus weakly 3-Tur\'an-good, improving a result from~\cite{GrzGyoSal22}.
 Gerbner~\cite{Ger22e} proved that if a graph $H$ is weakly $k$-Tur\'an-stable and has a unique proper $(k-1)$-coloring, and $H'$ is obtained from $H$ by adding edges (but no vertices) such that $\chi(H')=k-1$, then $H'$ is weakly $F$-Tur\'an-stable. Gerbner also obtained a general result that implies that if $H$ is weakly $3$-Tur\'an-good and we add pendant edges to $H$, then we obtain another weakly $3$-Tur\'an-good graph. 
Gerbner~\cite{Ger24c} showed that the graph we obtain by adding a pendant edge to each of two distinct vertices of a triangle is $4$-Tur\'an-good.

\subsection{Forbidden graphs without a color-critical edge}

By Theorem~\ref{gen-crit-edge}, we know that the extremal graph for $\ex(n,H,F)$ is a Tur\'an graph when $F$ has a color-critical edge. So now we turn our attention to $k$-chromatic graphs $F$ without a color-critical edge. The \textit{decomposition family} $\cD(F)$ of $F$ consists of every bipartite graph that can be obtained by removing the vertices of all but two color classes from a proper $k$-coloring of $F$. Let $\biex(n,F) := \ex(n,\cD(F))$ denote the maximum number of edges in an $n$-vertex graph that does not contain any member of the decomposition family of $F$. 
In particular, if $F$ has a color-critical edge, then $K_2$ is in the decomposition family of $F$ and so $\biex(n,F)=0$.
Gerbner~\cite{Ger22e} showed that if $H$ is weakly $F$-Tur\'an-stable, then 
\[
\ex(n,H,F)\le\cN(H,T)+\biex(n,F)\Theta(n^{|V(H)|-2})
\]
for some $n$-vertex complete $(k-1)$-partite  graph $T$. Moreover, if $H$ has a component $H'$ with an
\emph{almost proper coloring} (i.e.\ a vertex coloring with exactly one edge whose end-vertices are the same color) with at most $k$ colors, 
or $H'$ satisfies $\biex(n,H')=o(\biex(n,F))$, then
\[
\ex(n,H,F)=\cN(H,T)+\biex(n,F)\Theta(n^{|V(H)|-2})
\]
for some $n$-vertex complete $(k-1)$-partite graph $T$.
In particular, cliques satisfy this condition, which returns a theorem of Ma and Qiu~\cite{MaQiu18}.

The above results are paired with theorems that describe the structure of the corresponding extremal graphs. We do not state them precisely but present a brief summary. For each extremal graph $G$, there exists a $(k-1)$-partition of $V(G)$  and a set of vertices $B \subset V(G)$ of order $O(1)$ such that each graph in $\cD(F)$ that is inside a partition class of $G$ must share at least two vertices with $B$. Moreover, the vertices of $B$ are adjacent to $\Omega(n)$ vertices in each partition class of $V(G)$, while vertices outside $B$ are adjacent to $o(n)$ vertices in their partition class and all but $o(n)$ vertices in each other partition class. Subsection~\ref{stabmeth} includes a more detailed description of these extremal graphs.

If $F$ has a color-critical edge, then by Theorem~\ref{gen-crit-edge}, the extremal graph for $\ex(n,K_r,F)$ is the Tur\'an graph $T(n,\chi(F)-1)$ (assuming of course $r<\chi(F)$). For non-bipartite graphs $F$ without a color-critical edge, a theorem of Gerbner~\cite{Ger22f} identified a graph $H$ where we can describe a member of $\Ex(n, H,F)$ (for sufficiently large $n$) that is not a Tur\'an graph. 
 Suppose $F$ is a graph with chromatic number $k>2$. Define $\sigma(F)$ as the minimum number of vertices needed to delete from $F$ to obtain a graph with a smaller chromatic number. Let $H$ be the complete $(k-1)$-partite graph with parts of size $a$, for some large enough $a$. Then an extremal graph 
 for $\ex(n,H,F)$ is the join $K_{\sigma(F)-1}+T(n-\sigma(F)+1,k-1)$. This suggests that instead of Tur\'an-goodness, it may be more important to study which graphs satisfy 
 \[
 \ex(n,H,F)=\cN(H,K_{\sigma(F)-1}+T(n-\sigma(F)+1,k-1)),
 \]
 for sufficiently large $n$. 

Let us conclude with a summary of results where the extremal graph is known but is not a complete multipartite graph. The first such results (besides those for the ordinary Tur\'an function) are due to Gerbner and Patk\'os~\cite{GerPat21}. They showed that $\ex(n,K_r,2K_{k})=\cN(K_r,K_{1}+T(n-1,k-1))$ if $r< k$ and $n$ is large enough. They also showed that $\ex(n,K_r,B_{k,1})=\cN(K_r,T^+(n,k-1))$
where $B_{k,1}$ consist of two copies of $K_k$ that share exactly one vertex and
$T^+(n,k-1)$ is the graph that results from adding an edge to a smallest partition class of $T(n,k-1)$. Finally, they showed for all $s,t$ that $\ex(n,K_{s,t},B_{3,1})=\cN(K_{s,t},T^+)$ for some graph $T^+$ obtained by adding an edge to a complete bipartite graph. 
Note that if $s,t>1$, then the added edge is not in any copy of $K_{s,t}$ and so $K_{s,t}$ is weakly $B_{3,1}$-Tur\'an-good.
All of these results were improved by Gerbner~\cite{Ger22e}.

Let $H$ be a weakly $F$-Tur\'an-stable graph. For $n$ sufficiently large, Gerbner~\cite{Ger22e} determined $\ex(n,H,F)$ when $F$ has $s>1$ components, each with chromatic number $k$ and a color-critical edge, and any number of components with chromatic number less than $k$. Gerbner also determined $\ex(n,H,K_{1,a,\dots,a})$ for $H$ a forest.

Many of the results in this subsection extend theorems on $\ex(n,F)$ to $\ex(n,H,F)$ for certain weakly $F$-Tur\'an-stable graphs $H$. There are also theorems in \cite{Ger22e} that automatically give a generalized Tur\'an result from a corresponding ordinary Tur\'an result.
For example,
let $H$ and $F$ be graphs such that $\chi(H)<\chi(F)=k$ and $H$ is  
$F$-Tur\'an-stable.
If $\biex(n,F)=O(1)$ and $H$ is a forest that is $K_{k}$-Tur\'an-good, then, for $n$ sufficiently large,
$\Ex(n,H,F)\cap \Ex(n,F)\neq\emptyset$, i.e., there is an extremal graph common to both problems. If, in addition, $\chi(F)=3$ and there is a graph in $\Ex(n,F)$ that contains the Tur\'an graph $T(n,2)$, then $\Ex(n,H,F)\cap \Ex(n,F)\neq\emptyset$ holds for every bipartite $H$ that is 3-Tur\'an-stable and 3-Tur\'an-good, not just forests.
Gerbner~\cite{Ger24a} proved the following similar result
 for weak $F$-Tur\'an-stability.
Assume that $\biex(n,F)=O(1)$ and that the $n$-vertex $F$-free graph with the maximum number of edges is the Tur\'an graph with $\biex(n,F)$ edges added to exactly one partition class. Then the $n$-vertex $F$-free graph with the maximum number of copies of $S_r$ is a complete $(\chi(F)-1)$-partite graph with $\biex(n,F)$ edges added to exactly one partition class.


\section{The degenerate case}\label{section-degen}

The degenerate case is when $H$ and $F$ satisfy $\ex(n,H,F) = o(n^{|V(H)|})$. By Proposition~\ref{degen-prop} this occurs if and only if $F$ is a subgraph of a blowup of $H$. 
This section is organized according to the graph class to which the counting graph $H$ belongs. Let us emphasize that we include only results for the degenerate problem here but that there are many results in Section~\ref{turangood} that count copies of $H$ in the non-degenerate case.

\subsection{Counting triangles}\label{triangle}

Perhaps the most natural subgraph to count (after edges) is the triangle. In fact, the first arXiv posting of the paper~\cite{AloShi16} by Alon and Shikhelman dealt only with counting triangles and was titled ``Triangles in $H$-free graphs.''

We have already stated several results in Subection~\ref{history} that fit in this subsection. These include results of Bollob\'as and Gy\H ori~\cite{BolGyo08} on $\ex(n,K_3,C_5)$, of Gy\H ori and Li~\cite{GyoLi12} on $\ex(n,K_3,C_{2k+1})$ and several results counting cliques or cycles in general. In Subsection~\ref{diff} there appeared a result of Alon and Shikhelman~\cite{AloShi16} on $\ex(n,K_3,B_t)$.

Here we will list results for $H=K_3$ specifically, but note that later subsections consider classes of graphs $H$ that include triangles  (e.g.\ when counting cliques or cycles).
Kostochka, Mubayi, and Verstra{\"e}te~\cite{KosMubVer15} gave the general lower bound \[
\ex(n,K_3,F)= \Omega(n^{3-\frac{3|V(F)|-9}{e(F)-3}})\] 
for any $F$ with at least 4 edges using a simple random construction. 

\smallskip

The case when $F$ is a cycle is especially popular. Bounds given by Gy\H ori and Li~\cite{GyoLi12} were improved by F{\"u}redi and {\"O}zkahya~\cite{FurOzk17} and by Alon and Shikhelman~\cite{AloShi16} who proved $\ex(n,K_3,C_{2k+1})\le \frac{16(k-1)}{3}\ex(\lceil n/2\rceil, C_{2k})$. Gy\H ori and Li also showed the lower bound $\ex(n,K_3,C_{2k+1})\ge \frac{1}{2} \ex_{\textrm{bip}}(n,\cC^{\textrm{even}}_{\le 2k})$. It is unknown whether these bounds are of the same order of magnitude. Somewhat less attention has been paid to the case when $F$ is an even cycle---perhaps the close relation to the notoriously difficult $\ex(n,C_{2k})$ problem has made it less appealing. F{\"u}redi and {\"O}zkahya~\cite{FurOzk17} obtained the simple upper bound $\ex(n,K_3,C_{2k})\le \frac{2k-3}{3}\ex(n,C_{2k})$. Gishboliner and Shapira~\cite{GisSha18} gave the lower bound $\ex(n,K_3,C_{2k})= \Omega(\ex(n,\cC^{\textrm{even}}_{\le 2k}))$.

When $F=C_4$, Alon and Shikhelman~\cite{AloShi16} proved $\ex(n,K_3,C_4)=(1+o(1))\frac{1}{6}n^{3/2}$. 
For $F=C_5$, Bollob\'as and Gy\H ori~\cite{BolGyo08} gave a lower bound $(1+o(1))\frac{1}{3\sqrt{3}}n^{3/2}$. Their corresponding upper bound was improved by
Alon and Shikhelman~\cite{AloShi16}, Ergemlidze, Gy\H ori, Methuku, and Salia~\cite{ErgGyoMet19b}. A further improvement
is due to Ergemlidze and Methuku~\cite{ErgMet18}. Their proof uses a hypergraph Tur\'an result which has since been improved in~\cite{ConFoxSud20}
and implies the bound $(1+o(1))\frac{1}{2^{9/4}}n^{3/2}$.
Lv, He, and Lu~\cite{LvHeLu24} proved the slightly better bound $(1+o(1))\frac{1}{2\sqrt{6}}n^{3/2}$.

Hou, Yang and Zheng~\cite{HouYanZhe23} determined $\ex(n,K_3, tC_{2k+1})$ for sufficiently large $n$.
Alon and Shikhelman~\cite{AloShi16} gave bounds on $\ex(n,K_3,K_{s,t})$, including the order of magnitude when $t> (s-1)!$, and the asymptotics when $s=2$.
Gao, Wu, and Xue~\cite{GaoWuXue23} studied $\ex(n,K_3,F)$ where $F=F(p_1,\dots,p_k)$ is the \textit{generalized theta graph} consisting of internally vertex-disjoint paths of length $p_i$ between two fixed vertices. This graph generalizes both cycles and books. Depending on the value of each $p_i$, $\ex(n,K_3,F)$ can be quadratic, $O(n^{2-\alpha})$ for some $\alpha>0$ or between them: $n^{2-o(1)}\le \ex(n,K_3,F)\leq o(n^2)$. They characterized which holds for each value of $p_i$.

\smallskip

 A \textit{wheel} $W_k$ is a cycle $C_k$ joined with a universal vertex (i.e.\ adjacent to all other vertices). Let $\cW$ be the family of all wheels (where the cycle can have any length).
 In~\cite{Erd88}, Erd\H os mentions a conjecture of Gallai that $\ex(n,K_3,\cW)=\lfloor n^2/8\rfloor$. Let us describe a construction: take an $n$-vertex complete balanced bipartite graph and embed a matching in the partition class of size $\lceil n/2 \rceil$. 
Zhou~\cite{Zho95} disproved this conjecture by finding a slightly larger construction (with the same asymptotics) and proved an upper bound $\ex(n,K_3,\cW)\le \frac{1}{6}(n^2-n)$. 
Haxell~\cite{Hax95} proved $\ex(n,K_3,\cW)\le (1+o(1))\frac{1}{7}n^2$. 
F{\"u}redi, Goemans, and Kleitman~\cite{FurGoeKle94} found an even larger construction and showed $\ex(n,K_3,\cW)\ge \frac{1}{15}(2n^2+n)$, disproving Gallai's conjecture even in the asymptotic sense. They also improved the upper bound to $\ex(n,K_3,\cW)\le \frac{1}{7.02}n^2+5n$.

\smallskip

Mubayi and Mukherjee~\cite{MubMuk20} studied $3$-chromatic graphs $F$ that can be obtained by adding a universal vertex to a complete bipartite graph, an even cycle, or a path.
In particular, they showed $\ex(n,K_3,K_{1,s,t})=o(n^{3-3/s})$ if $s\le t$ and $\ex(n,K_3,W_{2k})=o(n^{2+1/k})$.
For the graph $H_k$ obtained by joining a universal vertex to a path $P_k$, they showed $\ex(n,K_3,H_k)=\Theta(n^2)$ and determined the aymptotics for $k=4$ and $k=5$. The case $k=4$ was improved to the exact result $\ex(n,K_3,H_4)=\lfloor n^2/8\rfloor$, for $n$ sufficiently large, by Gerbner~\cite{Ger22a}, and Mukherjee~\cite{Muk23} determined $\ex(n,K_3,H_4)$ for all $n$. Cui, Duan, and Yang~\cite{CuiDuaYan20} obtained bounds on the number of triangles when the forbidden graph consists of a path or an even cycle and multiple additional vertices, each adjacent to all the vertices of the original path or cycle.

Gerbner, Methuku, and Vizer~\cite{GerMetViz17} gave bounds on $\ex(n,K_3,F)$ when $F$ is disconnected.
Duan, Wang, and Yang~\cite{DuaWanYan18} determined $\ex(n,K_3,\cL(k))$, where $\cL(k)$ is the family of linear forests with at most $k$ edges (a linear forest is the vertex-disjoint union of paths).

Zhu, Cheng, Gerbner, Gy\H ori, and Hama Karim~\cite{ZhuCheGer22} determined $\ex(n,K_3,F_k)$ exactly, for sufficiently large $n$, where $F_k$ is the friendship graph composed of $k$ triangles that all share a single common vertex.  
Alon and Shikhelman~\cite{AloShi16} determined the graphs $F$ for which $\ex(n,K_3,F)=O(n)$. These are the so-called \textit{extended friendship graphs}, which contain the friendship graph $F_k$ as a subgraph for some $k\ge 0$, and no other cycles. 

One can view $F_k$ as an edge-blowup of the star $S_{k+1}$: replace every edge of $S_{k+1}$ by a triangle. Let $F^\triangle$ denote the graph we obtain if we replace each edge of $F$ with a triangle. Lv, Gy\H ori, He, Salia, Tompkins, Varga, and Zhu~\cite{LvGyoHe22b} found $\ex(n,K_3,F^\triangle)$ for $F$ a $P_4$ or $K_3$, when $n$ is sufficiently large. Liu and Song~\cite{LiuSon23} extended these results by determining $\ex(n,K_3,F^\triangle)$ asymptotically for every tree $F$, and exactly when $F$ is a path or a cycle of length at least $5$. They also obtained corresponding stability results.

\subsection{Counting cliques}\label{cliques}

The previous subsection dealt with the special case when $H$ is a triangle. We are not aware of many other results counting specific cliques beyond the following: Bollob\'as and Gy\H ori~\cite{BolGyo08} proved $\ex(n,K_4,C_5)\le (1+o(1))\frac18n^{3/2}$ in order to obtain an upper bound on $\ex(n,K_3,C_5)$ and Bayer, M\'esz\'aros, R\'onyai, and Szab\'o~\cite{BayMesRon19} showed (improving a lower bound due to Palmer, Tait, Timmons, and Wagner~\cite{PalTaiTim19}) that $\ex(n,K_4,K_{s,t})=\Theta(n^{4-{6}/{s}})$ if $t\ge (s-1)!+1$.

\smallskip

A very simple general bound by Gerbner and Palmer~\cite{GerPal18} is  $\ex(n,K_r,F)\le \ex(n,F)^{r/2}$.
Dubroff, Gunby, Narayanan, and Spiro~\cite{DubGunNar23} obtained the following general result. We say that a graph $F$ is \textit{2-balanced} if for every subgraph $F'$ of $F$, we have  $\frac{e(F')-1}{|V(F')|-2}\le \frac{e(F)-1}{|V(F)|-2}$. They showed that if $F$ is 2-balanced and has at least 2 edges, then $\ex(n,K_r,F)=\Omega(n^{2-\frac{|V(F)|-2}{e(F)-1}})$, for $2\le r<|V(F)|$.

Alon and Shikhelman~\cite{AloShi16} obtained multiple bounds on $\ex(n,K_r,K_{s,t})$ in different ranges of the parameters. Some of these bounds were improved by Ma, Yuan, and Zhang~\cite{MaYuaZha18}, Zhang and Ge~\cite{ZhaGe19}, Bayer, M\'esz\'aros, R\'onyai, and Szab\'o~\cite{BayMesRon19}. Summarizing, we have two regimes. 
If $r-1\le s\le t$, then $\ex(n,K_r,K_{s,t})=O(n^{r-{r(r-1)}/{2s}})$. This bound is sharp (in the order of magnitude) if $t$ is large enough or if $s\ge 2r-2$ and $t\ge (s-1)!+1$ or if $t\ge (s-1)!+1$ and $r=3$ or $r=4$.
On the other hand, if $r-1>s$, then  $\ex(n,K_r,K_{s,t})=O(n^{{(s+1)}/{2}})$. This is sharp (in the order of magnitude) if $s=2$ and $t\ge 2r-3$, by a result of Zhang and Ge~\cite{ZhaGe19} that shows $\ex(n,K_r,K_{2,t})=\Theta(n^{3/2})$.
For smaller $t$, Alon and Shikhelman~\cite{AloShi16} showed $\ex(n,K_r,K_{2,t})=\Omega(n^{4/3})$ by using a hypergraph Tur\'an result of Lazebnik and Verstra\" ete~\cite{LazVer03}. We remark that an improvement due to Haymaker, Tait, and Timmons~\cite{HayTaiTim24} implies $\ex(n,K_r,K_{2,t})=\Omega(n^{10/7})$ for all $r$, and $\ex(n,K_r,K_{2,t})=\Omega(n^{3/2-o(1)})$ for $r\le 6$. Balogh, Jiang, and Luo~\cite{BalJiaLuo24} gave bounds on $\ex(n,K_r,K_{s_1,\dots, s_r})$.

\smallskip

Gan, Loh, and Sudakov~\cite{GanLohSud15} showed that when $n$ is divisible by $k-1$, the extremal graph for $\ex(n,K_r,S_k)$ is made up of vertex-disjoint copies of $K_{k-1}$. They conjectured that for other values of $n$, $\ex(n,K_r,S_k)$ is given by vertex-disjoint copies of $K_{k-1}$ and a smaller clique, assuming $r>2$. 
This conjecture was proved by Chase~\cite{Cha19}, another proof is by Chao and Dong~\cite{ChaDon22}. A similar development occurs when paths are forbidden:
Luo~\cite{Luo18} proved $\ex(n,K_r,P_k)\le \frac{n}{k-1}\binom{k-1}{r}$, which is sharp when $n$ is divisible by $k-1$, as shown by vertex-disjoint copies of $K_{k-1}$. Chakraborti and Chen~\cite{ChaChe20b} extended this by proving that $\ex(n,K_r,P_k)$ is given by vertex-disjoint copies of $K_{k-1}$ and a smaller clique. The situation may be the same for all trees. Gerbner, Methuku, and Palmer~\cite{GerMetPal18} showed an identical upper bound for any tree if all its subtrees satisfy the Erd\H os-S\'os conjecture (Conjecture~\ref{erdsos}). This gives the exact value of $\ex(n,K_r,T)$ for such trees $T$ if $n$ is divisible by $k-1$. For other values of $n$, there is an additive constant error term.
Another tree where we know the exact value is the double-star $S_{a,b}$. The value of $\ex(n,K_r,S_{a,b})$ is given by vertex-disjoint cliques of order $|V(S_{a,b})|-1$ plus a smaller clique by a result of Gerbner~\cite{Ger21e}.

Luo~\cite{Luo18} showed $\ex(n,K_r,\cC_{\ge k})\le \binom{k-1}{r} \frac{n-1}{k-2}$, where $\cC_{\ge k}$ denotes the family of cycles of length at least $k$. 
This is sharp if $n$ is divisible by $k-2$. 
This implies the above result of Luo on $\ex(n,K_r,P_k)$.
Another proof of Luo's theorems was obtained by Ning and Peng~\cite{NinPen18}. Lv, Gy\H ori, He, Salia, Xiao, and Zhu~\cite{LvGyoHe22c} showed the following strengthening: $\ex(n,K_r,\cC^{\mathrm{odd}}_{\ge k})\le \binom{k-1}{r} \frac{n-1}{k-2}$ for $k\ge 13$ and $\ex(n,K_r,\cC^{\mathrm{even}}_{\ge k})\le \binom{k-1}{r} \frac{n-1}{k-2}$ for every $k$, where $\cC^{\mathrm{odd}}_{\ge k}$ and $\cC^{\mathrm{even}}_{\ge k}$ denote the set of odd, respectively, even, cycles of length at least $k$.

The family of \textit{generalized books} $\cB_{s,m,k}$ is the set of graphs consisting of $k$ cliques of order $m$ each sharing a common fixed set of $s$ vertices.  
Observe that there is only one graph in this family when $m=s+1$, and when $3=m=s+1$, this is the ordinary book graph $B_k$. Let $B_{s,m,k}$ denote the particular generalized book where the pairwise intersection of any two $m$-cliques is exactly an $s$-clique. Then $B_k=B_{2,3,k}$ and $\cB_{s,m,2}=\{B_{s,m,2},B_{s+1,m,2},\dots, B_{m-1,m,2}\}$.
Gowers and Janzer~\cite{GowJan20} showed $\Omega(n^{s-o(1)})\leq \ex(n,K_r,\cB_{s,r,2})\leq o(n^s)$, generalizing the result of Alon and Shikhelman~\cite{AloShi16} on $\ex(n,K_3,B_k)$. Their lower bound is given by a random geometric construction.
Gerbner~\cite{Ger21d} showed $\ex(n,K_r,B_k)=o(n^2)$ and $\ex(n,K_r,F_k)=O(n)$.
Liu and Wang~\cite{LiuWan21} studied $\ex(n,K_r,B_{s,r,2})$. They found its exact value when $s=0$ or $s=1$, and its order of magnitude in most of the other cases. Gerbner and Patk\'os~\cite{GerPat21b} studied $\ex(n,K_r,B_{s,m,2})$ and found its value exactly when $m=0$ or $m=1$ and $n$ is sufficiently large, and obtained bounds in the other cases. In particular, they determined $\ex(n,K_r,2K_m)$ exactly if $n$ is sufficiently large. 
Yuan and Wang~\cite{YuaWan22} improved the threshold on $n$ for $\ex(n,K_r,2K_m)$ and established the exact value of $\ex(n,K_3,2K_3)$. Gerbner~\cite{Ger23b} determined $\ex(n,K_r,tK_m)$ for all integers $r,m,t$ and sufficiently large $n$, improving the above results and those of Gerbner, Methuku, and Vizer~\cite{GerMetViz17}. Yuan and Peng~\cite{YuaPen25} established $\ex(n,K_r,B_{1,m,k})$ when $2\le k<m$, $3\le r<m$, and $n$ is sufficiently large. Hellier and Liu~\cite{HelLiu24} examined $\ex(n,K_r,\{B_{s_1,r,2},\dots,B_{s_i,r,2}\})$ in general (i.e.\ a family of forbidden intersections of $r$-cliques is given). Zhao and Zhang \cite{ZhaZha25} improved some of their results.

\smallskip

Gerbner, Methuku, and Vizer~\cite{GerMetViz17} gave general bounds on $\ex(n,K_r,tF)$ for arbitrary $F$, and obtained stronger bounds for the specific case when $F$ a cycle.
Gao, Wu, and Xue~\cite{GaoWuXue23}  determined $\ex(n,K_r,tF)$ when $F$ where $F=F(p_1,\dots,p_k)$ is the generalized theta graph consisting of internally vertex-disjoint paths of length $p_i$ between two fixed vertices, and $F$ has a color-critical edge, $3\le r\le t+1$, and $n$ is sufficiently large. Wang~\cite{Wan18} determined $\ex(n,K_r,M_k)$ for $M_k$ a matching of size $k$.  Zhu and Chen~\cite{ZhuChe21} (improving results of Zhu, Zhang, and Chen~\cite{ZhuZhaChe21}) determined $\ex(n,K_r,F)$ when $F$ is the vertex-disjoint union of paths of length at least 4 and $n$ is large enough. Chen, Yang, Yuan, and Zhang~\cite{CheYanYua24} determined $\ex(n,K_r,F)$ where $F$ is a linear forest with each component having an even number of vertices. Gao, Li, Lu, Sun, and Yuan \cite{GaoLiLu25} determined $\ex(n,K_r,tP_3)$ for sufficiently large $n$.
Liu and Yin~\cite{LiuYin25} determined $\ex(n,K_r,S_{k'}\cup S_k)$ if $r\ge 4$ and $\ex(n,K_r,3S_k)$ if $r\ge 5$, improving their earlier results in \cite{LiuYin23}. Khormali and Palmer~\cite{KhoPal20} determined $\ex(n,K_r,  tS_k)$ if $3 \leq r \leq k+t-2$, $k$ divides $n-t+1$ and $n$ large enough.

There are several other results for counting cliques that can be reformulated as generalized Tur\'an statements if we forbid multiple graphs.
Kirsch and Radcliffe~\cite{KirRad17b} studied $\ex(n,K_r,\{K_k,S_\ell\})$. Jin and Zhang~\cite{JinZha18} determined $\ex(n,K_r,\{tK_k,K_{k+1}\})$ and showed that the maximum is achieved by the complete $k$-partite graph with one partition class having size $t$ and the other partition classes having size as balanced as possible. 
Ma and Hou~\cite{MaHou23} gave upper and lower bounds on $\ex(n,K_r,\{F,M_\ell\})$ that differ by an additive term $O(1)$, and are sharp in some cases. Some further improvements were made by Zhu and Chen~\cite{ZhuChe23} and Xue and Kang~\cite{XueKan24}. Zhao and Lu~\cite{ZhaLu24} determined $\ex(n,K_r,\{\cC_{\ge k},M_{\ell}\})$ for most values of the parameters. Lu, Kang, and Xue \cite{LuKanXue25} extended this to every $r,\ell,k$ and sufficiently large $n$. Fang, Zhu, and Chen~\cite{FanZhuChe24} determined $\ex(n,K_r,\{K_k,P_\ell\})$ for some values of the parameters.

In 
the Part II of this survey~\cite{GerPal25} we will describe a connection between generalized Tur\'an problems and Berge hypergraph problems and how results on the Tur\'an number of Berge hypergraphs give bounds on $\ex(n,K_r,F)$ and for which graphs those results give the best known bounds. Let us briefly summarize the results most germane to this section. A hypergraph $\cH$ is a \textit{Berge copy} of a graph $H$ if there is a bijection between the edges of $H$ and the hyperedges of $\cH$ such that each edge is contained in its image. In other words, we can obtain a Berge copy of $H$ by arbitrarily expanding each edge of $H$ by additional vertices. 
Berge copies of cycles and paths have been studied for decades. Gerbner and Palmer~\cite{GerPal17} extended the definition to arbitrary graphs $H$. If a graph is $F$-free, then the hypergraph whose hyperedges are the vertex sets of copies of $K_r$ must not contain a Berge copy of $F$. This implies that the $r$-uniform hypergraph Tur\'an number of the family of Berge copies of $F$ is at least $\ex(n,K_r,F)$. An upper bound of $\ex(n,K_r,F)+\ex(n,F)$ was shown by Gerbner and Palmer~\cite{GerPal18}. The hypergraph Tur\'an number of the family of Berge copies of the cycle $C_{k}$ has attracted a great deal of attention (see~\cite{GyoLem12,GyoLem12b,GyoKatLem16}).
In particular, results in~\cite{GyoLem12b} imply that $\ex(n,K_r,C_{2k})=O(n^{1+1/k})$ and $\ex(n,K_r,C_{2k+1})=O(n^{1+1/k})$.

\subsection{Counting complete multipartite graphs}\label{multipa}


We have already seen results in Section~\ref{history} on counting copies of $K_{2,2}=C_4$ and we will state several more in Subsection~\ref{cycles}. An argument due to F\"uredi and West~\cite{FurWes01} is yet another example in which a generalized Tur\'an result is established to prove an ordinary Tur\'an result.
Let $H(s,t)$ denote the graph we obtain if we remove from $K_{s,s}$ the edges of a copy of $K_{t,t}$. In order to prove an upper bound on $\ex(n,H(s,t))$, they first proved $\ex(n,K_{r,r},H(s,t))\le \binom{n}{r}\max\{\binom{s-1}{r},\frac{1}{2}\binom{2s-r-1}{r}\}$. Grzesik, Janzer, and Nagy~\cite{GrzJanNag19} used this method to bound $\ex(n,F)$ for certain bipartite graphs $F$, and so obtained a bound on $\ex(n,K_{r,r},F)$.
Section~\ref{stars} deals with the case of counting copies of a star $S_{r+1}=K_{1,r}$.

\smallskip

When counting complete bipartite graphs, the typical forbidden graph is also complete bipartite.
Alon and Shikhelman~\cite{AloShi16} studied $\ex(n,K_{a,b},K_{s,t})$ in some ranges of the four parameters. Their results were extended or improved by Ma, Yuan, and Zhang~\cite{MaYuaZha18}, by Bayer, M\'esz\'aros, R\'onyai, and Szab\'o~\cite{BayMesRon19}, and by Gerbner, Methuku, and Vizer~\cite{GerMetViz17}, who also studied the more general $\ex(n,K_{a,b},kK_{s,t})$.
Gerbner and Patk\'os~\cite{GerPat21} examined $\ex(n,K_{a,b},K_{s,t})$ systematically and obtained various exact, asymptotic, and structural results. 

Let us summarize the state-of-the-art for $\ex(n,K_{a,b},K_{s,t})$. Without loss of generality, suppose that $a\le b$ and $s\le t$ and that $K_{a,b}$ does not contain $K_{s,t}$.
If $a>s$, then $b<t$ and the order of magnitude is unknown except when $s=1$ where we know the asymptotics as well as exact results for infinitely many $n$. In particular, if $2t-2$ divides $n$, then $\ex(n,K_{a,b},K_{1,t})=\cN(K_{a,b},\frac{n}{2t-2}K_{t-1,t-1})$.
If $a=s$, the exact result is trivial for $s=1$ and the asymptotics are known for $s=2$, but for larger $s$, the order of magnitude is unknown. It is possible that the general upper bound $O(n^s)$ from~\cite{GerMetViz17} is sharp in general here.
If $a<s$ and $b<s$, Alon and Shikhelman showed $\ex(n,K_{a,b},K_{s,t})=O(n^{a+b-ab/s})$. This was shown to be sharp if $t$ is large enough and in several other cases (see ~\cite{AloShi16,MaYuaZha18,BayMesRon19}).
It also remains open whether this bound is always sharp in this case.
Finally, if $a<s$ and $b\ge s$, we know the order of magnitude: $\ex(n,K_{a,b},K_{s,t})=\Theta(n^b)$. The asymptotics are known if $s<b$ or $s=b=t$: $\ex(n,K_{a,b},K_{s,t})=(1+o(1))\binom{s-1}{a}\binom{n}{b}$. In both cases, the lower bound is given by a \emph{complete split graph} $S(s-1,n-s+1)$, that is, a complete bipartite graph $K_{s-1,n-s+1}$ with all possible edges added to the class of size $s-1$.
This graph is actually an extremal graph if $s=a+1$ and $b\ge a+t-2\ge t$ or if $s>a+1$ and $a+b> s+t$. Further corresponding stability results are proved in~\cite{GerPat21}.

\smallskip

Gerbner, Nagy, and Vizer~\cite{GerNagViz20} determined the asymptotics of $\ex(n,K_{2,r},C_{2k})$. Gerbner~\cite{Ger21} studied $\ex(n,K_{2,r},F)$ for every $F$, and found the order of magnitude for every non-bipartite $F$, and when $F$ is a subgraph of $K_{2,k}$ for some $k$. 
Lu, Yuan, and Zhang~\cite{LuYuaZha21} established $\ex(n,K_{a,b},P_k)$,  $\ex(n,K_{a,b},M_k)$, and $\ex(n,K_{a,b},\cC_{\ge k})$. Gy{\H{o}}ri, He, Lv, Salia, Tompkins, Varga, and Zhu~\cite{GyoHeLv22} determined $\ex(n,K_{r,r},C_{2r+2})$.

Aside from the results above and those in other subsections, we are not aware of many results for counting complete multipartite graphs with more than two partition classes. Assume that $a_1\le \cdots \le a_r$ and $s_1\le \cdots \le s_r$. Ma, Yuan, and Zhang~\cite{MaYuaZha18} determined the order of magnitude of $\ex(n,K_{a_1,\dots,a_r},K_{s_1,\dots,s_r})$ when $a_1<s_1$, $a_i\le s_{i-1}$ for all $i\le r$, and $s_r$ is large enough. 
Gerbner and Patk\'os~\cite{GerPat21} determined the order of magnitude of $\ex(n,K_{a_1,\dots,a_r},K_{s_1,\dots,s_r})$ when $a_1<s_1$ and $a_i\ge s_{i-1}$ for all $i\le r$. Liu and Zhang~\cite{LiuZha24} obtained the exact value of $\ex(n,K_{a_1,\dots,a_r},P_k)$, $\ex(n,K_{a_1,\dots,a_r},M_k)$ and $\ex(n,K_{a_1,\dots,a_r},\cC_{\ge k})$. Gerbner~\cite{Ger23d} showed that for complete multipartite graphs $H$, $\ex(n,H,\{K_{k},M_{\ell}\})=\cN(H,G)$ for an $n$-vertex graph $G$ that is either a complete multipartite graph with at most $k-1$ partition classes and one partition class of order at least $n-\ell+1$, or a complete multipartite graph with at most $k-1$ partition classes spanning $2 \ell -1$ vertices, together with $n-2\ell+1$  isolated vertices.

\subsection{Counting cycles}\label{cycles}

We outlined results for counting triangles $C_3=K_3$ in Subsection~\ref{triangle} and some results for counting $C_4$ can be found in Subsection~\ref{multipa}. 

Let us now collect general results for $\ex(n,C_r,C_k)$. 
Gishboliner and Shapira~\cite{GisSha18} determined the order of magnitude of $\ex(n,C_r,C_k)$ when $r \neq 3$ (the case $r,k$ both even was proved independently in~\cite{GerGyoMet17}).

\begin{thm}[Gishboliner and Shapira]\label{cycle} For distinct $r,k$ we have 
\begin{displaymath}
\ex(n,C_r,C_k)=
\left\{ \begin{array}{l l}
\Theta_r(n^{r/2}) & \textrm{if\/ $r\ge 5$, $k=4$},\\
\Theta_r(k^{\lceil r/2\rceil}n^{\lfloor r/2\rfloor}) & \textrm{if\/ $k\ge 6$ even and\/ $r \ge 4$ or\/ $r,k$ odd and\/ $5\le r<k$},\\
\Theta_r(n^r) & \textrm{if\/ $r$ even and\/ $k$ odd or $r,k$ odd and\/ $5\le k<r$}.\\
\end{array}
\right.
\end{displaymath}
\end{thm}

Now to stronger results for $\ex(n,C_r,C_k)$ when $r>3$. Strong sharp results for $\ex(n,C_5,C_3)$ were given in Section~\ref{history}. 
Gerbner, Gy\H ori, Methuku, and Vizer~\cite{GerGyoMet17} proved $\ex(n,C_4,C_{2k})=(1+o(1))\frac{(k-1)(k-2)}{4}n^2$ for $k>2$.
Gy{\H{o}}ri, He, Lv, Salia, Tompkins, Varga, and Zhu~\cite{GyoHeLv22} determined $\ex(n,C_4,C_6)$.
Gerbner and Palmer~\cite{GerPal18} showed $\ex(n,C_r,K_{2,t})=\left(\frac{1}{2r}+o(1)\right)(t-1)^{r/2}n^{r/2}$ (excluding the empty case $r=4$, $t=2$), in particular this gives the asymptotics of $\ex(n,C_r,C_4)$. 

Gerbner, Gy\H ori, Methuku, and Vizer~\cite{GerGyoMet17} obtained several results when a family of cycles is forbidden. This includes the asymptotics for the number of four-cycles and the order of magnitude of the number of six-cycles in graphs when any family of cycles is forbidden. Zhu, Gy{\H{o}}ri, He, Lv, Salia, and Xiao~\cite{ZhuGyoHe22} determined the number of copies of $C_4$ and $C_5$ if all cycles of length at least $k$ are forbidden.

Using the Erd\H os Girth Conjecture (Conjecture~\ref{erdgirth})
Solymosi and Wong~\cite{SolWon17} gave a construction establishing $\ex(n, C_{2r}, \cC_{\le 2k}) =\Omega(n^{2r/k})$ and $\ex(n, C_{2r}, \cC_{\le 2k+1}) =\Omega(n^{2r/k})$ for $r>k$. 
 A corresponding upper bound from \cite{GerGyoMet17} when $k+1$ divides $2r$ is as follows. Let $G$ be an $n$-vertex graph with no cycle of length at most $2k$. We count $2r$-cycles by fixing each of the $(k+1)$st edges. By the Bondy-Simonovits Even Cycle Theorem (Theorem~\ref{BS-even-cycle}), $e(G) = O(n^{1+1/k})$. So we can select the at most $\frac{2r}{k+1}$ edges of $G$ in $O(n^{2r/k})$ ways. Now observe that between the end-vertices of a pair of these edges, the forbidden cycle condition ensures that there is at most one path of length $\ell$ for all $\ell \leq k$. Therefore, these edges correspond to a constant number of cycles $C_{2r}$ and so $\ex(n, C_{2r}, \cC_{\le 2k}) =O(n^{2r/k})$.
Since the Girth Conjecture holds for $k = 2,3,5$, this gives the correct order of magnitude for several cases. The constant factor in the lower bound for $\ex(n, C_{6}, \cC_{\le 5})$ and $\ex(n, C_{8}, \cC_{\le 7})$ was specifically
improved in~\cite{YanSunZha23}.

Gy\H ori, Salia, Tompkins, and Zamora~\cite{GyoSalTom18} established asymptotics for $\ex(n,C_{r},P_k)$, and, for $n$ large enough, the exact value of $\ex(n,C_4,P_k)$ and $\ex(n,C_4,\cC_{\ge k})$.
Gerbner and Palmer~\cite{GerPal18} determined the graphs $F$ that force $\ex(n,C_r,F)=O(n)$ for every $r$, extending the corresponding result for triangles by Alon and Shikhelman~\cite{AloShi16}.

In Section~\ref{history} we saw that sometimes results for $\ex(n,H,F)$ are used as lemmas toward results on $\ex(n,F)$. In these situations, typically $H$ is a very simple graph.
Erd\H os and Simonovits~\cite{ErdSim69} gave an upper bound on the Tur\'an number of graphs built as follows: take an arbitrary bipartite graph $F$ and a complete bipartite $K_{t,t}$, and connect each vertex in one partition class of $F$ to each vertex in one partition class of $K_{t,t}$, then do the same for the other partition class of $F$ and $K_{t,t}$. This gives a bipartite graph $F'_t$. In order to bound $\ex(n,F'_t)$, they gave upper bounds on $\ex(n,P_3,F'_t)$ and $\ex(n,C_4,F'_t)$.

\subsection{Counting stars}\label{stars}

A bound of $\ex(n,S_{s+1},K_{s,t}) \leq (t-1)\binom{n}{s}$ was given as a step in the proof of the K\H ov\'ari-S\'os-Tur\'an theorem (see  Section~\ref{history}). 
Several results on counting multipartite graphs (Subsection~\ref{multipa}) also fit in this subsection. 
Gy\H ori, Salia, Tompkins, and Zamora~\cite{GyoSalTom18} determined the exact value of $\ex(n,S_r,P_k)$ and $\ex(n,S_r,\cC_{\geq k})$. Huang and Qian~\cite{HuaQia22} determined $\ex(n,S_r,F)$ when $F$ is a linear forest.
Gerbner~\cite{Ger20} showed $\ex(n,S_r,C_4)=\cN(S_r,F)$ where $F$ is an $n$-vertex friendship graph. Note that if $r\ge 4$, then $\cN(S_r,F)=\cN(S_r,S_n)=\binom{n-1}{r-1}$, since many of the edges of the friendship graph $F$ are not contained in an $S_r$.

Given an $n$-vertex graph $G$ with degrees $d_1\le d_2\le \cdots \le d_n$, we have $\cN(S_r,G)=\sum_{i=1}^n \binom{d_i}{r-1}$. This is closely related to the sum $\sum_{i=1}^n d_i^{r-1}$, which is a well-studied quantity, especially in chemical graph theory; see the survey~\cite{AliGutMil18}. There, the sum $R_\alpha(G)=\sum_{i=1}^n d_i^{\alpha}$ is called the \textit{general zeroth order Randic index}, but the terms \textit{first general Zagreb
index} and \textit{variable first Zagreb index} are also used in the literature. Section 2.4 of the survey~\cite{AliGutMil18} discusses the maximum value of this index among $n$-vertex $F$-free  graphs. This line of research was initiated by Caro and Yuster~\cite{CarYus00}.

A particularly interesting special case is $r=3$, i.e., $\sum_{i=1}^n d_i^{2}$. This sum is dubbed the \emph{first
Zagreb index} and is surveyed in \cite{BorDasFur17}. Observe that $\sum_{i=1}^n d_i^{2} =2\cN(S_3,G)+2e(G)$. Thus, if the number of edges and the number of $3$-vertex stars are maximized by the same graph, then this graph also maximizes the first Zagreb index. 
Gerbner~\cite{Ger24a} studied the connection between the zeroth order Randic index and counting stars, and showed the following, extending this simple observation.
Denote by $R_\alpha(n,F)$ the largest value of the general zeroth order Randic index $R_\alpha(G)$ among $n$-vertex $F$-free graphs $G$. 
Gerbner~\cite{Ger24a} used the following proposition to show that several results for $R_r(n,F)$ are consequences of results on $\ex(n,S_r,F)$, and to obtain new results on $R_r(n,F)$ for various graphs $F$. 

\begin{proposition}[Gerbner~\cite{Ger24a}]\label{csill} 

\begin{enumerate}[(i)]
    \item\label{item-1} $\ex(n,S_r,F)=\left(\frac{1}{(r-1)!}+o(1)\right)R_{r-1}(n,F)$. Moreover, the same graphs are asymptotically extremal for both functions.

    \item\label{item-2} For any graph $G$, there are positive reals $w_i=w_i(r)$ such that \[R_r(G)=\sum_{i=1}^r w_i \cdot \cN(S_i,G).\]

\end{enumerate}

\end{proposition}

In Part II of the survey we will further explore the connection between the Randic index and the generalized Tur\'an problem for stars. But let us now highlight several instances where $R_{r-1}(n,F)$ gives the best-known bounds on $\ex(n,S_r,F)$.

F\"uredi and K\"undgen~\cite{FurKun06} described how the order of magnitude of $R_r(n,F)$, and so that of $\ex(n,S_r,F)$, depends on the order of magnitude of $\ex(n,F)$. 
More precisely, assume that $\ex(n,F)=\Theta(n^{2-\beta})$. If $r\beta<1$, then $R_r(n,F)=\Theta(n^{r-r\beta})$, if $r\beta>1$, then $R_r(n,F)=\Theta(n^{r-1})$, while if $r\beta=1$, then $\Omega(n^{k-1})\le R_r(n,F)\le O(n^{k-1}\log n)$. They conjectured that the $\log n$ factor can be removed from the upper bound. Gao, Liu, Ma, and Pikhurko \cite{GaoLiuMa24} proved this for several classes of well-studied bipartite graphs.

Nikiforov~\cite{Nik09} showed $R_r(n,C_{2k})=(k-1+o(1))n^r$. 
Lan, Liu, Qin, and Shi~\cite{LanLiuQin19} gave results on $R_r(n,F)$ for certain forests $F$. They also considered the \textit{broom} $B(k,\ell)$, which is obtained
from $P_k$ by adding $\ell$ new leaves connected to a penultimate vertex of the
path. They determined $R_r(n,B(k,\ell))$ for $r \geq 2$, $5\le k\le 7$, and $n$ sufficiently large. Extensions are given by Wang and Yin~\cite{WanYin22} for $k=8$ and by Gerbner~\cite{Ger24b} for all $k$ who also determined $\ex(n,S_r,B(k,\ell))$ for $r\ge 2$ and sufficiently large $n$.
Brooks and Linz~\cite{BroLin23} obtained several exact and asymptotic results on $R_r(n,F)$ in connection to spectral Tur\'an problems, see Subsection~\ref{spectral}.

\subsection{Counting other forests}\label{trees}

Let us continue with paths. 
Gy\H ori, Salia, Tompkins, and Zamora~\cite{GyoSalTom18} determined the asymptotics of $\ex(n,P_r,P_k)$. They also gave the exact value of $\ex(n,P_3,P_k)$, $\ex(n,P_4,P_k)$, $\ex(n,P_5,P_k)$ and $\ex(n,P_6,P_7)$ for $n$ large enough.
Gishboliner and Shapira~\cite{GisSha18} found the order of magnitude of $\ex(n,P_r,C_k)$, while Gerbner, Gy\H ori, Methuku, and Vizer~\cite{GerGyoMet17} determined its asymptotics when $k$ is odd or $r<2k$ is odd.
Gerbner and Palmer~\cite{GerPal18} showed $\ex(n,P_r,K_{2,t})=\left(1+o(1)\right)\frac{1}{2}(t-1)^{(r-1)/2}n^{(r+1)/2}$.
Erd\H os and Simonovits~\cite{ErdSim69} proved bounds on $\ex(n,P_3,F'_t)$ in order to bound $\ex(n,F'_t)$ (see Section~\ref{cycles} for details and the definition of $F_t'$). 
In~\cite{ErdSim82}, while bounding $\ex(n,\{C_4,C_5\})$, an upper bound on $\ex(n,P_4,\{C_4,C_5\})$ is established (and also on the number of three-edge walks in $\{C_4,C_5\}$-free graphs).

Gerbner~\cite{Ger21} studied $\ex(n,T,K_{2,t})$, when $T$ is a tree and determined the order of magnitude for every tree and characterized which trees satisfy $\ex(n,T,K_{2,t})=(1+o(1))\cN(T,F(n,t))$, where $F(n,t)$ is the F\"uredi graph (see Subsection~\ref{algebra}).
Gerbner~\cite{Ger20} showed that if $F$ is not a forest, then $\ex(n,M_r,F)=(1+o(1))\frac{1}{r!}\ex(n,F)^r$. Gerbner~\cite{Ger24b}  extended this result to trees $F$ 
and characterized the forests that also attain this value. 
Both $\ex(n,M_2,S_4)$ and $\ex(n,M_2,P_4)$ are determined in~\cite{Ger20}, and in~\cite{Ger24d}, these results were extended to $\ex(n,M_2,S_k)$ and $\ex(n,M_2,P_k)$, for sufficiently large $n$. 
Gerbner~\cite{Ger21e} gave bounds on $\ex(n,S_{a,b},S_{c,d})$.
 Cambie, Verclos, and Kang~\cite{CamVerKan19} considered the case when stars $S_k$ are forbidden. They observed that $\ex(n,S_r,S_k)$ is given by any $(k-2)$-regular graph, and if $T$ is a tree with maximum degree at most $k-1$ and diameter $d$, then $\ex(n,T,S_k)$ is given by any $(k-2)$-regular graph with girth at least $d+1$.

\subsection{Counting other graphs}\label{other}

We conclude Section~\ref{section-degen} with results that count graphs $H$ that do not otherwise fit in the previous subsections.
Wang~\cite{Wan18} determined $\ex(n,H,M_k)$ for $H=S(s,t)$, the complete split graph, that is, a complete bipartite graph $K_{s,t}$ with all edges added to the partition class of size $s$. 
Gerbner, Methuku, and Vizer~\cite{GerMetViz17} initiated the study of counting disconnected graphs $H$ and determined the asymptotics of $\ex(n,r K_3,tK_3)$. Gy\H ori, Salia, Tompkins, and Zamora~\cite{GyoSalTom18} gave bounds on $\ex(n,M_r\cup K_3,P_{k})$, $\ex(n,M_r\cup 2K_3,P_{k})$, and $\ex(n,M_r\cup K_4,P_{k})$ as a lemma toward bounding $\ex(n,P_r,P_k)$.
Among other results, Gerbner~\cite{Ger20} studied $\ex(n,H,F)$ for each of the 100 pairs of graphs $H,F$ on at most four vertices without isolated vertices.

\smallskip

Let us continue with results where $H$ is more general. 
 Alon and Pudl{\'a}k~\cite{AloPud01} showed that projective norm graphs provide a lower bound for a certain problem in Ramsey theory. 
 In their proof they found the order of magnitude of the number of copies of certain graphs $H$ in these norm graphs. As the norm graphs are $K_{t,t!+1}$-free, this gives a lower bound on $\ex(n,H,K_{t,t!+1})$.
Recall that a graph is \emph{$3$-degenerate} if any non-empty subgraph of it has a vertex of degree at most $3$.
Bayer, M\'esz\'aros, R\'onyai, and Szab\'o~\cite{BayMesRon19} obtained a lower bound $\Omega(n^{|V(H)|-{e(H)}/{s}}) \leq \ex(n,H,K_{s,t})$ for any $3$-degenerate graph $H$ if $t>(s-1)!$ and $s\ge 4$. As we will see in Subsection~\ref{algebra}, Ma, Yuan, and Zhang~\cite{MaYuaZha18} proved that for \emph{any} graph $H$, if $t$ is large enough (depending on $e(H)$ and $s$), then $\ex(n,H,K_{s,t}) = \Omega(n^{|V(H)|-{e(H)}/{s}})$.

Kirsch and Nir~\cite{KirNir23} gave the asymptotics of $\ex(n,H,S_k)$ if $H$ has a universal vertex.
Kirsch \cite{Kir25} established the asymptotics of $\ex(n,H,\{S_r,K_k\})$ if $H$ has a universal vertex. 

Gerbner~\cite{Ger24d} determined $\ex(n,H,M_{k})$ for graphs $H$ with a vertex cover of size at most $k-1$ (and $n$ sufficiently large). 
Moreover, Gerbner determined $\ex(n,H,M_{k})$ if $H$ has minimum degree at least $k$, or if exactly one vertex has degree $k-1$ and the others have degree at least $k$.

Gerbner~\cite{Ger20} showed that if $F$ is obtained from $K_k$ by adding a pendant edge,  
then $\ex(n,H,F)=\ex(n,H,K_k)$ for every connected graph $H\neq K_k$, provided that $n$ is large enough. It seems likely that further such results can be obtained where some small perturbation of $F$ does not change $\ex(n,H,F)$ too much.

Gir{\~a}o, Hunter, and Wigderson~\cite{GirHunWig24} showed that for any triangle-free graph $H$, there is a constant $\alpha_H$ such that $\ex(n,H,H[k])<n^{|V(H)|-\alpha_H/k}$ for the blowup $H[k]$.

\smallskip

Finally, we state results where $F$ is fixed and $H$ is more general. Typically, in this setting, the goal is simply to determine the order of magnitude of $\ex(n,H,F)$.
Alon and Shikhelman~\cite{AloShi16} determined the order of magnitude of $\ex(n,H,F)$ when both $H$ and $F$ are trees, and studied the case $H$ bipartite and $F$ a tree. 
Gy\H ori, Salia, Tompkins, and Zamora~\cite{GyoSalTom18} gave an upper bound on the order of magnitude of $\ex(n,H,F)$ for every graph $H$ and tree $F$.
Extending these results, Letzter~\cite{Let18} established the correct order of magnitude of $\ex(n,H,F)$ for every graph $H$ and tree $F$. 

Let us describe Letzter's construction, which is based on the same partial blowup idea discussed after Proposition~\ref{gissha}. Given a graph $H$, a set $U\subset V(H)$ and an integer $t$, the \textit{$(U,t)$-blowup} of $H$ is the graph $H'$ we obtain by replacing each vertex $v\in V(H)\setminus U$ by $t$ vertices $v_1,\dots,v_t$. More precisely, vertices $u,u'\in U$ form an edge in $H'$ if and only if they form an edge in $H$, and for vertices $u\in U, v\not\in U$ and $i\le t$, we have that $uv_i$ is an edge of $H'$ if and only if $uv$ is an edge of $H$, and for vertices $v,v'\not\in U$ and $i,j\le t$, we have that that $v_iv'_j$ is an edge of $H'$ if and only if $vv'$ is an edge of $H$. Given graphs $F$ and ($F$-free) $H$, let $r(H,F)$ denote the maximum number of components in $H\setminus U$, over sets $U\subset V(H)$ such that the $(U,|V(F)|)$-blowup of $H$ is $F$-free. Clearly, $\ex(n,H,F) = \Omega(n^{r(H,F)})$. Indeed, adding further identical vertices does not help to create a copy of $F$, and so the $(U,n/|V(H)|)$-blowup of $H$ is also $F$-free and has at most $n$ vertices. 
Observe that each component in $H\setminus U$ is duplicated at least $n/|V(H)|$ times in the $(U,n/|V(H)|)$-blowup.



 Gerbner~\cite{Ger23b,Ger23d} determined the order of magnitude of $\ex(n,H,tK_k)$ for all $H$, $t$, and $k$ and $\ex(n,H,\{F,M_{k}\})$ for all $H$, $F$, and $k$. In both of these arguments, the construction that gives the lower bound is also a `partial' blowup of $H$, where we blow up the largest number of vertices possible while avoiding the forbidden subgraphs.

\section{Methods}\label{meth}

In this section, we give an overview of several methods that have been applied to generalized Tur\'an problems. Most of these techniques are familiar in extremal graph theory. In each subsection, we give a simple representative application of the method discussed.

\subsection{Regularity lemma}\label{regularity}

Szemer\'edi's Regularity Lemma~\cite{Sze76} revolutionized how we view the structure of graphs in general.
 Informally, the lemma states that \emph{every} large graph has a special structure; it admits a partition into a fixed number of parts such that almost every pair of parts is regular in some sense. 
See~\cite{KomSim96,KomShoSim00} for not very recent surveys on the regularity lemma and its applications. An especially nice treatment of the regularity lemma, its proof, and associated lemmas appear in lecture notes by Yufei Zhao~\cite{zhao-lecture-notes}.

To state the lemma formally, we need a few definitions.
Let $X,Y\subset V(G)$ be disjoint sets of vertices of a graph $G$.  Let $d(X,Y)=\frac{e(X,Y)}{|X||Y|}$ be the \textit{density} between the pair $X,Y$ where $e(X,Y)$ counts the number of edges between $X$ and $Y$.
We say that a pair $X,Y$ of disjoint sets of vertices is \textit{$\varepsilon$-regular} if for any $X'\subset X$ and $Y'\subset Y$ with $|X'|\ge \varepsilon |X|$ and 
$|Y'|\ge \varepsilon |Y|$, we have $|d(X,Y)-d(X',Y')|\le \varepsilon$. Informally, aside from negligible small subsets, every pair of subsets has roughly the same density. This property is satisfied, with high probability, by random graphs, for example.

\begin{definition}
A vertex-partition of a graph $G$ into parts $V_0,V_1,\dots,V_r$ is \emph{$\varepsilon$-regular}, if
\begin{itemize}
\item $|V_0|\le \varepsilon|V(G)|$,
\item $|V_1|=|V_2|=\dots =|V_r|$,
\item all but at most $\varepsilon r^2$ of the pairs $V_i,V_j$ with $1\le i<j\le r$ are $\varepsilon$-regular.

\end{itemize}
\end{definition}

\begin{thm}[Szemer\'edi's Regularity Lemma] For  $\varepsilon>0$ and $m$, there exist $n_0$ and $M=M(\varepsilon,m)$ such that every graph on $n\ge n_0$ vertices admits an $\varepsilon$-regular partition $V_0,V_1,\dots, V_r$ with $m\le r \le M$.
\end{thm}

This statement is often interpreted as ``every large graph is random-like.''
Note that the above theorem is only meaningful for dense graphs. 
Indeed, if $d(X,Y)\le \varepsilon^3$, then $X'\subset X$, $Y'\subset Y$, $|X'|\ge \varepsilon |X|$ and 
$|Y'|\ge \varepsilon |Y|$ imply $d(X',Y')\le \varepsilon$, thus $X,Y$ is an $\varepsilon$-regular pair. If $G$ has $n$ vertices and $\varepsilon^4n^2$ edges, consider a partition satisfying the first two properties. Then a pair $V_i,V_j$ with $i,j\ge 1$ that is not $\varepsilon$-regular has at least $\varepsilon^3(n-\varepsilon n)^2/r^2$ edges. As an edge counts only for one such pair, there can be at most $\varepsilon r^2$ pairs which are not $\varepsilon$-regular, thus the third property is automatically satisfied.

Given $\varepsilon>0$ and $\delta>0$, a graph $G$ and its $\varepsilon$-regular partition $V(G)=V_0\cup V_1\cup\cdots\cup V_r$, its \textit{$(\varepsilon,\delta)$-reduced graph} or \textit{cluster graph} has vertex set $v_1,\dots,v_r$ and $v_iv_j$ is an edge if and only if $V_i,V_j$ is an $\varepsilon$-regular pair with density at least $\delta$.

Now we state three lemmas that leverage the existence of an $\varepsilon$-regular partition. The last two are evidently relevant to the subgraph counting problem. We use the statements from~\cite{zhao-lecture-notes}.

\begin{lemma}[Embedding Lemma]
Let $H$ be an $r$-partite graph with maximum degree $\Delta$. Let $G$ be a graph with vertex partition $V_1,V_2,\dots, V_r$ such that $|V_i| \geq \frac{1}{\varepsilon}|V(H)|$. If every pair $V_i,V_j$ is $\varepsilon$-regular and has density $d(V_i,V_j) \geq 2\varepsilon^{1/\Delta}$, then $G$ contains a copy of $H$.
\end{lemma}

\begin{lemma}[Counting Lemma]
    Fix $\varepsilon >0$. Let $H$ be a graph on vertex set $\{1,2,\dots,r\}$.
    Let $G$ be a graph with vertex partition $V_1,V_2,\dots, V_r$
    such that $V_i,V_j$ is $\varepsilon$-regular whenever $ij \in E(H)$. Then the 
 number of copies $\cN(H,G)$ of $H$ is within $\varepsilon e(H) |V_1|\cdots |V_r|$ of 
    \[
    \left(\prod_{ij \in E(H)} d(V_i,V_j)\right)  \left(\prod_{i=1}^r |V_i|\right).
    \]    
\end{lemma}

\begin{lemma}[Removal Lemma]\label{removal_lemma}\label{removallemma} For each $\varepsilon>0$ and graph $H$, there exists $\delta>0$ such that every $n$-vertex graph $G$ with  $\cN(H,G) <\delta n^{|V(H)|}$ can be made $H$-free by removing  less than $\varepsilon n^2$ edges.
\end{lemma}

Let us illustrate how the regularity lemma can be used for generalized Tur\'an problems with a sketch of the proof of Theorem~\ref{genESS}, which we restate here for convenience.

\generalboundthm*

\begin{proof}[Proof sketch]
Let $G$ be an $F$-free graph, let $\varepsilon$ be small enough, let $m$ be large enough, and apply the Regularity Lemma. Consider the resulting $\varepsilon$-regular partition and remove all the edges within a $V_i$, all the edges incident to $V_0$, all the edges between $V_i$ and $V_j$ if the pair $V_i,V_j$ is not $\varepsilon$-regular or if its density is less than $\varepsilon+\delta$ for some $\delta$. By choosing the constants carefully, we remove only $o(n^2)$ edges. Obviously, there are only $o(n^{V(H)})$ copies of $H$ that we deleted this way, as an edge is contained in $O(n^{|V(H)|-2})$ copies of $H$. Thus, if $G$ had more copies of $H$ than the bound in the theorem statement, then more than $\ex(n,H,K_k)$ copies of $H$ remain, and therefore there is a copy of $K_k$ in the resulting graph. Observe that this implies the existence of a copy of $K_k$ in the $(\varepsilon,\delta)$-reduced graph. Applying the Embedding Lemma to the $k$ corresponding parts, we obtain a copy of $F$ in $G$, a contradiction.
\end{proof}

\subsection{Probabilistic method}\label{probab}

Probabilistic tools are among the most powerful in extremal graph theory; see the monograph by Alon and Spencer~\cite{AloSpe16} for a comprehensive overview.
For generalized Tur\'an problems, this method has found only limited use thus far. We will give examples where an upper and a lower bound is proved using probabilistic methods. We start with a more precise version of Proposition~\ref{degen-prop}.

\begin{thm}[Alon and Shikhelman~\cite{AloShi16}]\label{asblow} Let $H$ be an $r$-vertex graph. Then $\ex(n,H,F)=\Theta(n^r)$ if and only if $F$ is not a subgraph of a blowup of $H$. Otherwise $\ex(n,H,F)=O(n^{r-\varepsilon})$ for some $\varepsilon=\varepsilon(H,F)>0$.
\end{thm}

\begin{proof} If $F$ is not a subgraph of any blowup of $H$, then the blowup $H[\lfloor n/r\rfloor]$ is $F$-free and has $\Omega(n^r)$ copies of $H$.

Now assume that $F$ is contained in some blowup $H[t]$. Let $G$ be an $n$-vertex $F$-free graph and put $m:=\cN(H,G)$. Consider a random partition of $V(G)$ into $r$ partition classes $V_1,\dots,V_r$ and let $u_1,\dots,u_r$ be the vertices of $H$. Call a copy of $H$ in $G$ {\it nice} if the vertex corresponding to $u_i$ is in $V_i$ for all $i$. The expected number of nice copies of $H$ is $m/r^r$, since each vertex $u_i$ belongs to part $V_i$ with probability $1/r$. Therefore, there is a partition of $V(G)$ with at least $m/r^r$ nice copies of $H$.
Fix such a partition and construct a $r$-uniform hypergraph $\cH$ on vertex set $V(G)$  whose hyperedges are nice copies of $H$ in $G$. Observe that $\cH$ is necessarily $r$-partite. Let $\cH'$ be the complete $r$-uniform $r$-partite hypergraph with partition classes of size $t$. Observe that if $\cH$ contains $\cH'$, then the vertices of $\cH'$ span a copy of $H[t]$ in $G$, which implies that $G$ contains $F$, a contradiction. 
A result of of Erd\H os~\cite{Erd64b} implies that the (hypergraph) Tur\'an number of an $r$-uniform $r$-partite graph is $O(n^{r-\varepsilon})$. Therefore, $\cH$ has $O(n^{r-\varepsilon})$ hyperedges and so $G$ has $m=r^rO(n^{r-\varepsilon})=  O(n^{r-\varepsilon})$ copies of $H$.
\end{proof}

Using a standard first-moment argument, Gerbner and Palmer~\cite{GerPal18} gave the following general lower bound on $\ex(n,H,F)$. 

\begin{proposition}
Let $F$ and $H$ be graphs with $f:=|V(F)|$, $h:=|V(H)|$, such that $e(F) > e(H)$. Then
\[
\Omega\left(n^{h - \frac{e(H)(f-2)}{e(F)-e(H)}}\right) \leq \ex(n,H,F).
\]    
\end{proposition}

\begin{proof}
Consider a random graph from $G(n,p)$ with edge probability 
 \[
 p=cn^{-\frac{f-2}{e(F)-e(H)}},
 \] 
 where $c=h^{h(e(F)-e(H))}+1$. The expected number of copies of $F$ is at most $n^fp^{e(F)}$ and the expected number of copies of $H$ is at least $\left(\frac{n}{h}\right)^hp^{e(H)}$. Deleting an edge from each copy of $F$ destroys at most $n^fp^{e(F)}n^{h-2}$ copies of $H$. In the resulting $F$-free graph, the expected number of copies of $H$ is at least 
 \[
 \left(\frac{n}{h}\right)^hp^{e(H)}-n^fp^{e(F)}n^{h-2}=\Omega\left(n^{h - \frac{e(H)(f-2)}{e(F)-e(H)}}\right).
 \]
 Therefore, there exists an $n$-vertex graph such that after removing an edge from each copy of
$F$, we are left with at least $\Omega\left(n^{h - \frac{e(H)(f-2)}{e(F)-e(H)}}\right)$ copies of $H$.    
\end{proof}

\subsection{Stability method}\label{stabmeth}

Sometimes in extremal combinatorics it is easier to prove a stability result than an exact result. In this situation, the structural information provided by stability may help to obtain an exact result for sufficiently large $n$. This idea is behind some of the most general exact results for generalized Tur\'an problems. 
Here we discuss the (weakly) $F$-Tur\'an-stable property which is a straightforward generalization of Erd\H os-Simonovits stability (Theorem~\ref{esstab}). Subsection~\ref{turstab} lists results on weakly $F$-Tur\'an-stable graphs, so here we focus on how stability implies exact results, which was mostly omitted from the discussion in Subsection~\ref{turstab}.

Let $H$ and $F$ be graphs with $\chi(H)<\chi(F)=k$. Recall 
\begin{enumerate}
    \item[(1)] $H$ is \textit{$F$-Tur\'an-stable} if any $n$-vertex $F$-free graph $G$ with $\ex(n,H,F)-o(n^{|V(H)|})$ copies of $H$ has edit distance $o(n^2)$ from $T(n,k-1)$.
    \item[(2)] $H$ is \textit{weakly $F$-Tur\'an-stable} if any $n$-vertex $F$-free graph $G$ with $\ex(n,H,F)-o(n^{|V(H)|})$ copies of $H$ has edit distance $o(n^2)$ from some $n$-vertex complete $(k-1)$-partite graph $T$.
    \item[(3)] $H$ is \textit{$F$-Tur\'an-good} if $\ex(n,H,F)=\cN(H,T(n,k-1))$ (for $n$ sufficiently large).  
    \item[(4)] $H$ is \textit{weakly $F$-Tur\'an-good} if $\ex(n,H,F)=\cN(H,T)$ for some $n$-vertex complete $(k-1)$-partite graph $T$ (for $n$ sufficiently large).

\end{enumerate}

The first stability result for generalized Tur\'an problems is Theorem~\ref{maqiustabi} due to Ma and Qiu~\cite{MaQiu18} which states that if $r < k$, then $K_r$ is $F$-Tur\'an-stable for every graph $F$ with $\chi(F)=k$. 
This theorem was used to obtain an exact result: $K_r$ is $F$-Tur\'an-good if $F$ has a color-critical edge (Theorem~\ref{gen-crit-edge}).

Let us sketch a basic example of how to use stability to obtain an exact result. Assume that $H$ is $K_{k}$-Tur\'an-stable and let $G$ be an $n$-vertex $K_{k}$-free graph with $\ex(n,H,F)$ copies of $H$. Pick a complete $(k-1)$-partite graph $T$ with minimum edit distance from $G$. Let $A_1,\dots, A_{k-1}$ be the partition classes of
$T$ and consider them as a vertex partition of $G$. The edit distance condition guarantees that in $G$, if a vertex has $\Omega(n)$ neighbors in its own partition class, then it has $\Omega(n)$ neighbors in every partition class. It can be shown that each vertex has degree at least roughly $\frac{k-2}{k-1}n$, and thus there are two types of vertices: those that are joined to all but $o(n)$ neighbors in every partition class besides their own, and those that have $\Omega(n)$ neighbors in every partition class. Now, let $U$ denote the set of vertices in $G$ that are incident to $\Theta(n)$ non-edges between partition classes, hence $|U|=o(n)$. If $uv$ is an edge of $G$ inside of a partition class, then $u$ and $v$ cannot have common neighbors in each other partition class because $G$ is $K_{k}$-free. This implies that the vertices of an edge inside a partition class are both in $U$. Therefore, there are $\Omega(n|U|)$ missing edges between partition classes and at most $\binom{|U|}{2}=o(n|U|)$ edges inside partition classes. The number of copies of $H$ containing an edge in any partition class is $o(n^{|V(H)|-2}|U|^2)$, while the number of copies of $H$ that are removed from $T$ by removing the missing edges between partition classes in $G$ is $\Omega(n^{|V(H)|-1}|U|)$. Therefore, $T$ contains more copies of $H$ than $G$ does, a contradiction. It is not hard to modify this proof to instead forbid a $k$-chromatic graph $F$ with a color-critical edge.
Therefore, we have our exact result: the Tur\'an graph $T(n,k-1)$ is the extremal graph for $\ex(n,H,K_k)$, for $n$ sufficiently large.

An important advantage of the stability method is the following simple observation: if $H$ is $K_{k}$-Tur\'an-stable, then $H$ is $F$-Tur\'an-stable for every $F$ with $\chi(F)=k$. Indeed, in this case, $\ex(n,K_{k},F)=o(n^{k})$ by Proposition~\ref{degen-prop}, and so we can remove all copies of $K_{k}$ by deleting $o(n^2)$ edges by the Removal :emma (Lemma~\ref{removallemma}). In this way, the number of copies of $H$ decreases by $o(n^{|V(H)|})$, and so the stability assumption still holds.

After the results of Ma and Qiu~\cite{MaQiu18}, there were a number of papers that used variants of this idea for several special cases. Hei, Hou, and Liu~\cite{HeiHouLiu21} applied it more broadly and Gerbner~\cite{Ger22b} obtained the following general formulation.

\begin{thm}[Gerbner~\cite{Ger22b}]\label{stabcrit}
Let $H$ and $F$ be graphs such that $\chi(F)>\chi(H)$ and $F$ has a color-critical edge. If $H$ is weakly $F$-Tur\'an-stable, then $H$ is weakly $F$-Tur\'an-good. 
\end{thm}

This method also works if we go beyond the Tur\'an graph. The first such result is due to Gerbner~\cite{Ger21e}. Observe that at most one edge can be added to a complete bipartite graph without creating a copy of $F_2$ (the bowtie graph, i.e., the $5$-vertex friendship graph). Using a stability approach, it was shown that such a graph contains the maximum number of copies of the double-star $S_{a,b}$ among the $n$-vertex $F_2$-free graphs. Later, Gerbner~\cite{Ger22e} formalized what stability gives in general, in the following way. Recall that $\sigma(F)$ denotes the minimum number of vertices whose deletion reduces the chromatic number of $F$.

\begin{thm} [Gerbner~\cite{Ger22b}]
Let $H$ and $F$ be graphs such that
 $k=\chi(F)>\chi(H)$ and $H$ is weakly $F$-Tur\'an-stable. Then for every $n$-vertex $F$-free graph $G$ with $\cN(H,G)=\ex(n,H,F)$, there is a $(k-1)$-partition of $V(G)$ into $A_1,\dots,A_{k-1}$, a constant $K_F$ and a set $B$ of at most $(k-1)(\sigma(F)-1)K_F$ vertices such that every vertex of $B$ is adjacent to $\Omega(n)$ vertices in each partition class, and every vertex of $A_i\setminus B$ is adjacent to $o(n)$ vertices in $A_i$ and all but $o(n)$ vertices in $A_j$ with $j\neq i$.
\end{thm}

This structural result was used in~\cite{Ger22e} and~\cite{Ger22f} to obtain several exact results. Note that the statement above can be strengthened by replacing the assumption $\cN(H,G)=\ex(n,H,F)$ with $\cN(H,G)\ge\ex(n,H,F)-o(n^{|V(H)|-1})$.

\subsection{Flag algebras}\label{flag}

Razborov's~\cite{Raz07} \emph{flag algebra method} is a formal calculus, applicable to problems in extremal combinatorics, that incorporates several often used proof ideas. An advantage of this method is that it yields an almost fully-automated way to obtain non-trivial bounds for many problems. The method was used to solve generalized Tur\'an problems in~\cite{Grz12,HatHlaKra13,LidPfe18,LidMur20,MurNir21,HeiHou22}. 
Moreover, a basic flag algebra idea is to count induced subgraphs, which is highly relevant to the subject of this survey.

 Razborov~\cite{Raz07} formulated the theory of flag algebras in the language of finite model theory, where the models can be graphs, hypergraphs, directed graphs, tournaments, $F$-free graphs, etc. For a simpler introduction focused on graphs, see~\cite{DeCFilSat16,Wan19}. We only give some definitions here, focusing only on graphs, and present some of the steps that distinguish the actual method. Further details can be found in many of the articles cited above in this subsection. There is a computer program specifically developed for flag algebra calculations by Vaughan, called \textit{Flagmatic}. The newest version seems to be at https://github.com/retagaine/Flagmatic, and was updated by Wang in 2019, in conjunction with his BsC thesis~\cite{Wan19}.

We begin by showing an elementary example of a generalized Tur\'an proof where we count induced subgraphs. Afterwards, we give another elementary, but somewhat more complicated, example and then conclude with a more formal discussion.

\begin{proposition}[Gerbner~\cite{Ger20}]\label{primiflag}
$\ex(n,P_4,K_4)=\cN(P_4,T(n,3))$.
\end{proposition}

\begin{proof} 
Let $G$ be a $K_4$-free graph on $n$ vertices.
An induced subgraph of $G$ on four vertices is one of $M_2,P_4,C_4,B_2,T_1$ (where $T_1$ denotes the triangle with a pendant edge) or is $M_2$-free and $P_4$-free .
Let $\cN^*(H,G)$ denote the number of induced copies of $H$ in $G$. 
It is easy to see that 
\[
\cN(M_2,G)=2\cN^*(B_2,G)+2\cN^*(C_4,G)+2\cN^*(T_1,G)+\cN^*(P_4,G)+\cN^*(M_2,G)
\] 
and \[
\cN(P_4,G)=6\cN^*(B_2,G)+4\cN^*(C_4,G)+2\cN^*(T_1,G)+\cN^*(P_4,G).\]
This implies that $\cN(P_4,G)\le 2\cN(M_2,G)+2\cN^*(B_2,G)$. The same equations hold for $T(n,3)$ in place of $G$, and obviously $\cN^*(T_1,T(n,3))$, $\cN^*(P_4,T(n,3))$, and $\cN^*(M_2,T(n,3))$ are all $0$. This implies $\cN(P_4,T(n,3))= 2\cN(M_2,T(n,3))+2\cN^*(B_2,T(n,3))$.

As $G$ is $K_4$-free, we have $\cN^*(B_2,G)=\cN(B_2,G)$. A result
from~\cite{GyoPacSim91} implies $2\cN(B_2,G)\le 2\cN(B_2,T(n,3))$ 
and a result from~\cite{Ger20} gives $\ex(n,M_2,K_4)=\cN(M_2,T(n,3))$.
Putting this all together yields
\[
\cN(P_4,G)\le 2\cN(M_2,T(n,3))+2\cN^*(B_2,T(n,3))=\cN(P_4,T(n,3)). \qedhere
\]
\end{proof}

The main ingredient in the above proof is pure luck---everything lines up perfectly. The only induced subgraph remaining in the final inequality is $B_2$, which happens to be the only graph where we had available the bound in the non-induced version. The number of copies of $M_2$ and $P_4$ are maximized by the same graph, which conveniently does not contain induced copies of the other three graphs. It might seem unlikely to find a comparable solution for other problems, but this is not the case: Qian, Xie, and Ge~\cite{QiaXieGe21} showed that $\ex(n,P_4,K_k)=\cN(P_4,T(n,k-1))$ for $k \geq 5$ using a similar argument. Their proof is slightly more complicated in that they had to account for $\cN(K_4,G)$, and  
in addition to equations for $\cN(M_2,G)$ and $\cN(P_4,G)$ they had to express $\cN^*(K_3 \cup K_1,G)$ in terms of induced subgraphs on four vertices. They also showed that $\ex(n,P_5,K_k)=\cN(P_5,T(n,k-1))$ by examining 13 different $5$-vertex induced subgraph to create 13 equations. This suggests that even though other problems can be solved with this approach, it likely becomes impractical as the order of the graph we want to count increases. But there is another problem, even for smaller instances: how do we find which equations and inequalities we should use? In the above examples we only used linear equations and inequalities, but perhaps in other cases, quadratic or more complicated equations would help.

Flag algebras also suffer from the first problem, but they can handle the other main issue by a systematic approach, and they are amenable to computer assistance. Flag algebras are also much more general, both in the sense that they are designed to handle extremal problems beyond only graphs and in the sense that they are more versatile (e.g.\ enumeration of labeled objects). 
 However, we also lose something, as they typically prove only asymptotic bounds. Usually, an exact bound can only be obtained afterward, using additional ideas. Moreover, the bounds are asymptotic in the sense that we bound the limit of a normalized version of the desired quantity, e.g., $\ex(n,H,F)/\binom{n}{|V(H)|}$. This limit is $0$ in several cases (if and only if $F$ is a subgraph of a blowup of $H$ by Proposition~\ref{degen-prop}) in which case the flag algebra method cannot usually be applied.

The basic idea is to obtain inequalities as in the proof of Proposition~\ref{primiflag} and apply optimization techniques to obtain a bound which may or may not be strong. We have to decide one parameter: in the proof of Proposition~\ref{primiflag}, to count copies of $P_4$, we considered the $4$-vertex induced subgraphs of $G$. However, we can consider larger subgraphs as well. In general, this approach may give better bounds, but at the cost of increased running time.

Let us describe another example. Bondy~\cite{Bon97} (attributing his proof strategy to Goodman~\cite{goodman}) gave a counting proof of Mantel's theorem, $\ex(n,K_3)\le n^2/4$. 
Razborov identified this argument in~\cite{Raz07} as one of the predecessors of the flag algebra method. Consider the $3$-vertex induced subgraphs of an $n$-vertex triangle-free graph $G$. Denote by $x$ the number of empty 3-sets, by $y$ the number of 3-sets with exactly one edge, and by $z$ the number of 3-sets with exactly two edges. Then $x+y+z=\binom{n}{3}$, $y+2z=e(G)(n-2)$ and $z=\sum_{v\in V(G)} \binom{d(v)}{2}$. 
Applying the Cauchy-Schwartz inequality gives $z \geq \frac{2e(G)^2}{n} - e(G)$.
From here, a straightforward computation completes the proof. 

Let us now modify this argument to illustrate deeper flag algebra ideas. We will establish an asymptotic version of Mantel's theorem: $\ex(n,K_3)/\binom{n}{2} \to 1/2$ as $n \to \infty$. This is the usual introductory example of flag algebras in the literature; see, e.g.,~\cite{DeCFilSat16}. Consider the $3$-vertex induced subgraphs; each has at most 2 edges. Averaging over all 3-sets gives the bound $\ex(n,K_3)/\binom{n}{2}\le 2/3$. If we do the same for the $5$-vertex induced subgraphs, we obtain the bound $\ex(n,K_3)/\binom{n}{2}\le 3/5$. Seven vertices gives the bound $\ex(n,K_3)/\binom{n}{2}\le 4/7$, and so on. Increasing the order of the induced subgraphs leads to better and better bounds. Unfortunately, this brute-force approach quickly becomes intractable.
Let us hint at how to overcome this difficulty and obtain the desired bound. There is a property of Bondy's proof that we have not used in this asymptotic argument: in Bondy's third equation we actually count the copies of the cherry $P_3$ by first selecting its middle vertex. 
Essentially we are now examining partially-labeled graphs. Let $G(v)$ denote the graph $G$ with vertex $v$ labeled. Instead of $\binom{d(v)}{2}$ we measure the `density' in $G(v)$ of the cherry with the middle vertex labeled. The key observation is that this density is the square of the density of the single edge with one vertex labeled.
From here, 
by summing over all $v \in V(G)$ and applying the Cauchy-Schwartz inequality, one can relate the density of the (unlabeled) cherry to the edge density of $G$.

Now we are ready to introduce flag algebras more formally. We restrict ourselves to graphs, but these definitions can be given for other structures analogously. A \textit{type of size $k$} is a labeled $F$-free graph on $k$ vertices $\{1,\dots,k\}$. An embedding of a type 
$\sigma$ to a graph $H$ is an injective function $\theta: \{1,\dots,k\}\rightarrow V(H)$ that preserves edges and non-edges. A \textit{$\sigma$-flag} is a pair $(H,\theta)$ where $\theta$ is an embedding of $\sigma$ to $H$. Two $\sigma$-flags $(H,\theta)$ and $(H',\theta')$ are isomorphic if there is an isomorphism between $H$ and $H'$ that also preserves the labels. Let $(G,\theta)$, $(H_1,\theta_1)$ and $(H_2,\theta_2)$ be $\sigma$-flags such that $|V(H_1)|+|V(H_2)|\le |V(G)|$. Let us choose two disjoint sets
$X_1$ and $X_2$
uniformly at random from $V(G)\setminus \mathrm{Im}(\theta)$ such that $|X_i|=|V(H_i)|-k$ for $i=1,2$. The \textit{density of $(H_1,\theta_1)$ and $(H_2,\theta_2)$ in $G$} is the probability that $X_i\cup \mathrm{Im}(\theta)$ restricted to $G$, with the labels of $\theta$, is isomorphic to $(H_i,\theta_i)$ for both $i=1$ and $i=2$.

One can show that this joint density is approximately equal to the product of their individual densities. To be more precise, consider the formal linear combinations of $\sigma$-flags. Razborov~\cite{Raz07} showed that after dealing with some technical details, we can define multiplication and obtain an algebra. Consequently, the function $\ex(n,H,F)$ is extended to a function from this algebra to the non-negative reals.

We conclude with a few words about the remaining steps. Counting copies of a single flag gives a single inequality, and counting copies of several flags gives several inequalities.
The goal is to find the best bound on $\ex(n,H,F)$ by combining such inequalities. Unfortunately, this is infeasible unless the number of flags is very small, thus we can only hope to obtain \emph{some} bound. This is done via semidefinite programming. The basic idea is to take the densities of the flags examined as a vector $p$ and consider an arbitrary positive semidefinite matrix $M$. Calculating $p^TMp$, which is nonnegative, will give new bounds. Fortunately, the mature theory of semidefinite programming can be leveraged here.

\subsection{Spectral methods}\label{spectral}

The \emph{adjacency matrix} $A=A(G)$ of a graph $G$ with vertices $v_1,\dots, v_n$ is an $n\times n$ matrix, where the entry $a_{ij}$ is 1 if $v_iv_j$ is an edge of $G$ and 0 otherwise. The \emph{spectrum} of a matrix is the set of its eigenvalues. The largest eigenvalue $\lambda(G)$ is of particular importance in extremal problems. A standard question in the subject of \emph{spectral extremal graph theory} (see~\cite{Nik11}) seeks to maximize the spectral invariants of the adjacency matrix of a graph $G$ with specified conditions (e.g.\ $F$-free). As $\lambda(G)$ is at least the average degree, results for spectral versions of Tur\'an-type problems can imply classical results.
For example, Nikiforov proved the following \emph{spectral Tur\'an's theorem} which implies Tur\'an's theorem.

\begin{theorem}[Nikiforov~\cite{Nikiforov}]
	 If $G$ is $K_{k}$-free, then
	\[
	\lambda(G) \leq \sqrt{2 \frac{k-2}{k-1}e(G)}.
	\]
\end{theorem}

Another connection is that the number of walks of length $r$ in an $n$-vertex graph $G$ is at most 
$n\lambda(G)^{r}$. The number of paths of length $r$ is at most half of the number of walks of length $r$, since every path contains the two walks starting from each of the path's two end-vertices. Thus, an upper bound on $\lambda(G)$ gives an upper bound on the number of paths of length $r$ and so results that determine the maximum value of $\lambda(G)$ for an $n$-vertex $F$-free graph $G$ immediately imply bounds on $\ex(n,P_{r+1},F)$.

This simple observation was first used by Gy\H ori, Salia, Tompkins, and Zamora~\cite{GyoSalTom18} to give an asymptotically sharp upper bound when $F$ is a path via a spectral result of Nikiforov~\cite{Nik10}. Then Gerbner, Gy\H ori, Methuku, and Vizer~\cite{GerGyoMet17} gave an asymptotically sharp upper bound when $F$ is an even cycle, using another result of Nikiforov~\cite{Nik08}.
Gerbner and Palmer~\cite{GerPal20} established asymptotics for non-bipartite $F$.

\begin{proposition}\label{path-asym}
   If $F$ is $k$-chromatic with $k>2$, then 
   \[
   \ex(n,P_r,F)=(1+o(1))\cN(P_r,T(n,k-1)).
   \]
\end{proposition}

\begin{proof}[Proof sketch.]
By the discussion above,
the number of walks of length $r-1$ in $G$ is at most $\lambda(G)^{r-1}n$.
Nikiforov~\cite{Nik09b} and independently Babai and Guiduli~\cite{Gui96} proved that if $F$ has chromatic number $k$ and $G$ is an $n$-vertex $F$-free graph, then $\lambda(G)\le (1-\frac{1}{k-1}+o(1))n$. Therefore, the number of paths $P_r$ of length $r-1$ is
\[
\ex(n,P_r,F)\le \frac{1}{2}\left(1-\frac{1}{k-1}+o(1)\right)^{r-1}n^{r}.
\]
A greedy count of copies of $P_r$ in the Tur\'an graph $T(n,k-1)$ completes the argument.
\end{proof}

Brooks and Linz~\cite{BroLin23} applied spectral results to bound $R_r(n,F)$, which gives asymptotic bounds on $\ex(n,S_r,F)$ (see Subsection~\ref{stars} for more details). 
In particular, they used a result by Hofmeister~\cite{Hof88} that $R_2(G)\le n\lambda(G)^2$. This implies that if the extremal graph for the spectral Tur\'an problem is $K_{k,n-k}$, or obtained from $K_{k,n-k}$ by adding each edge inside the smaller partition class (resulting in a complete split graph), or by adding each edge inside the smaller partition class and one more edge inside the larger partition class, then this graph is asymptotically extremal for $\ex(n,S_r,F)$. Brooks and Linz~\cite{BroLin23} mention several examples, including paths, and a long list of examples can be found in~\cite{ByrDesTai24}.

\subsection{Progressive induction}

In this subsection we describe Simonovits' technique of progressive induction (see \cite{Sim68}). There are several theorems in extremal graph theory where ordinary induction almost works: it is easy to prove the induction step if $n$ is larger than some threshold $n_0$, but the base step for $n_0$ is not practical (or does not hold).

We begin with an informal description of the basic idea.
Suppose that we want 
to prove that two functions satisfy $\alpha(n)\le \beta(n)$, for $n$ large enough. Imagine that we cannot prove the bound for $n=n_0$, so ordinary induction does not work. But we can prove the induction step: if $n$ is large enough and we increase $n$ by 1 (or 2, etc.), then $\alpha$ does not increase more than $\beta$ does. The key is to argue a stronger induction step: if $\alpha$ increases by the same amount as $\beta$, then we have some additional information that can be used to prove the result we need. When we have these parts of an argument, it is easy to combine them into a proof: for $n\le n_0$, we may have $\alpha(n)>\beta(n)$, but the difference is bounded by a constant $c$ depending on $n_0$. As $n$ increases, each time we apply the stronger (progressive) induction step, either $\alpha$ and $\beta$ increase by the same amount and our additional information completes the proof, or their difference decreases by at least one. In that case, after $c$ steps, we have $\alpha(n)\le\beta(n)$.

Let us describe a simple concrete example. Simonovits used progressive induction to prove the Critical Edge Theorem (Theorem~\ref{simo}). For simplicity, we give the argument for $C_5$ only. The statement is that for $n$ large enough, $\ex(n,C_5)=e(T(n,2))=\lfloor n^2/4\rfloor$, and moreover, the Tur\'an graph is the unique extremal graph (the statement obviously does not hold for $n\le 4$, and it is easy to see that it does not hold for $n=5$). Let $G$ be an $n$-vertex $C_5$-free graph. If $G$ does not contain a $K_{3,3}$, then we are done by the Erd\H os-Stone-Simonovits theorem or the K\H ov\'ari-S\'os-Tur\'an Theorem (as $n$ is large enough). Therefore, $G$ contains a $K_{3,3}$, call it $K$, which is necessarily an induced $K_{3,3}$. It is easy to see that any vertex in $G-K$ cannot be adjacent to vertices in both partition classes of $K$. The induction step will be from $n-6$ to $n$: there are exactly $9$ edges $K$, at most $3(n-6)$ edges between $K$ and $G-K$, and induction provides an upper bound on the number of edges in $G-K$. From here, a simple calculation shows that if the statement holds for $n-6$, then it also holds for $n$. But this is not enough since we have no base step. Instead, it is necessary to examine the situation when we have equality. If, indeed, every vertex in $G-K$ is connected to 3 vertices of $K$, then the vertices of $G-K$ connected to the same partition class of $K$ form an independent set, and thus $G-K$ is bipartite. Therefore, $G$ itself is bipartite and thus has at most $\lfloor n^2/4\rfloor$ edges. 

Continuing, let $n_0$ be the minimum integer such that any $n_0$-vertex graph with more than $\lfloor n_0^2/4\rfloor$ edges contains a copy of $K_{3,3}$. Put 
\[
c:=\max \{e(G)-\lfloor n^2/4\rfloor: \text{ $n\le n_0$ and $G$ is an $n$-vertex $C_5$-free graph}\}.
\]
 We claim that if $n\ge n_0+6c$, then every $n$-vertex $C_5$-free graph $G$ has at most $\lfloor n^2/4\rfloor$ edges. Indeed, consider a graph $G_1$ on $n$ vertices such that $n_0\le n\le n_0+6$. If $G_1$ has more than $\lfloor n^2/4\rfloor$ edges, then, by the definition of $n_0$, $G_1$ contains a $K_{3,3}$, denoted $K$. If $G_1$ has $3n-9$ edges incident to the vertices of $K$, then $G_1$ is bipartite by the argument in the previous paragraph. Otherwise, $G_1$ has at most $\lfloor n^2/4\rfloor+c-1$ edges, as the subgraph on the remaining $n-6$ vertices has at most $\lfloor (n-6)^2/4\rfloor+c$ edges. Similarly, a graph $G_2$ on at most $n_0+12$ vertices has at most $c-2$ edges more than the corresponding (bipartite) Tur\'an graph, and so on. This difference in the number of edges eventually vanishes at $n_0+6c$, which completes the proof. 
 
We now modify this proof to show that $\ex(n,P_3,C_5)=\cN(P_3,T(n,2))$ (a result of Gerbner~\cite{Ger20}) and $\ex(n,C_4,C_5)=\cN(C_4,T(n,2))$ (a result of Gerbner and Palmer~\cite{GerPal20}). If $n$ is large enough and $G$ is an $n$-vertex $C_5$-free graph, then again we can assume $G$ contains an induced $K_{3,3}$, denoted $K$. Let $G-K$ denote the graph obtained by removing $K$. As before, no vertex in $G-K$ is connected to vertices in both partition classes of $K$. When counting edges, this meant that the number of edges incident to $K$ was at most $3n-9$. 
Moreover, when there are exactly $3n-9$ such edges, then the entire graph $G$ is bipartite. Analogously, we must show that the number of copies of $P_3$ and $C_4$ containing at least one vertex of $K$ is maximized by the Tur\'an graph.

Let us begin with $P_3$. There are 18 copies of $P_3$ in $K$. For each fixed vertex $v$ in $G-K$, to count copies of $P_3$ containing $v$ and two vertices in $K$, we must either pick two of the at most three neighbors of $v$ in $K$ or a neighbor of $v$ in $K$ and an edge inside $K$ incident to that neighbor. In this way, we have at most $12(n-6)$ such copies of $P_3$. Now fix two vertices $u$ and $v$ in $G-K$. If $u$ and $v$ are adjacent and do not have a common neighbor in $K$, then we can count the copies of $P_3$ containing $u$, $v$ and a vertex of $K$ by picking a neighbor of $u$ or $v$ in $G$ in at most 6 ways. If $u$ and $v$ are adjacent and have a common neighbor in $K$, then neither of them has a further neighbor in $K$, so there are 3 ways to obtain a copy of $P_3$.
If $u$ and $v$ are not adjacent, then we can choose a common neighbor to obtain a $P_3$ that contains them and a vertex of $K$ in at most 3 ways. This implies that the number of such copies of $P_3$ is at most 
\[
3\binom{n-6}{2}+3e(G-K)\le 3\binom{n-6}{2}+3e(T(n-6,2)).
\] 
We have equality in each step above for the (bipartite) Tur\'an graph (using the above result on the number of edges and that $n-6$ is large enough). Furthermore, as when counting edges, we have equality if every vertex of $G-K$ is adjacent to 3 vertices in the same partition class of $K$. Therefore, $G$ is bipartite and has at most $\cN(P_3,T(n,2))$ copies of $P_3$. The same argument as when counting edges with constant
\[
c:=\max \{\cN(P_3,G)-\cN(P_3,T(n,2)): \text{ $n\le n_0$ 
and $G$ is an $n$-vertex $C_5$-free graph}\}
\] will complete the proof.

For $\ex(n,C_4,C_5)$, first count the copies of $C_4$ containing 3 vertices of $K$. For any vertex of $G-K$, we have to pick two of its at most three neighbors in $K$, and then one of their three common neighbors in $K$. Therefore, there are at most $9(n-6)$ such 4-cycles. 
Now count the 4-cycles with two vertices $u$ and $v$ in $G-K$. If $u$ and $v$ are adjacent, they cannot be adjacent to distinct vertices in the same partition class of $K$. Therefore, they either have one common neighbor in the same partition class of $K$ and no other neighbor in $K$ (in which case no $C_4$ contains $u$, $v$, and two vertices from $K$) or they have neighbors in different partition classes of $K$ and there are at most nine $4$-cycles of this type.
If $u$ and $v$ are not adjacent, then we have to pick two of their at most three common neighbors in $K$, so there are at most $3\binom{n-6}{2}+6e(G-K)e$ such 4-cycles. The number of copies of $C_4$ that contain one vertex from $K$ is at most $3\cN(P_3,G-K)$, as we have to choose a common neighbor of the end-vertices of a $P_3$. There are 27 copies of $C_4$ inside $K$. As before, we have equality in each step for the (bipartite) Tur\'an graph using the earlier results on $\ex(n,C_5)$ and $\ex(n,P_3,C_5)$. Repeating the argument above with constant
\[c:=\max \{\cN(C_4,G)-\cN(C_4,T(n,2)): \text{ $n\le n_0$ and $G$ is an $n$-vertex $C_5$-free graph}\}\]
will complete the argument.

As the above examples suggest, applying progressive induction becomes much more complicated as we count larger graphs. It is no surprise then that the first application was to count $P_3$: Gerbner~\cite{Ger20} showed that $P_3$ is $F$-Tur\'an-good for any $3$-chromatic graph with a critical edge. Gerbner~\cite{Ger20c} used progressive induction to show that if $H$ is the complete $k$-partite graph $K_{b,\dots,b}$ and $F$ is the complete $(k+1)$-partite graph $K_{1,a,\dots,a}$ where $b>2a-2$ and $n$ is large enough, then $\ex(n,H,F)=\cN(H,T(n,k))$, i.e., $H$ is $F$-Tur\'an-good.

Finally, having outlined the method, let us give more formal statements, beginning with the original very general version from~\cite{Sim68}.

\begin{lemma}[Progressive Induction Lemma]\label{sim-pi}
Let $\cA=\bigcup_{i=1}^\infty \cA_n$ be a set such that $\cA_n$ are disjoint finite subsets of $\cA$. Let $B$ be a condition or property defined on $\cA$ (i.e., the elements of $\cA$ may satisfy or not satisfy $B$). Let $\Delta : \cA \to \mathbb{Z}_{\geq 0}$ be such that: 
\begin{enumerate}
    \item[(a)] if $a\in\cA$ satisfies $B$, then $\Delta(a)=0$.
    \item[(b)]  There is an $n_0$ such that if $n>n_0$ and $a\in \cA_n$, then either $a$ satisfies $B$ or there exists $n/2 < n' < n$ and $a' \in \cA_{n'}$ such that $\Delta(a)<\Delta(a')$. 
\end{enumerate}
Then there exists $N$ such that for every $n>N$, every element of $\cA_n$ satisfies $B$.
\end{lemma}

Next is a version tailored to generalized Tur\'an problems, based on~\cite{Ger20c}. For a graph $G$ and a subgraph $G'$, let $\cN_I(H,G,G')$ denote the number of copies of $H$ in $G$ that contain at least one vertex from $G'$.

\begin{lemma}\label{progin} Fix graphs $F$ and $H$ and
let $\cG_n$ be a family of $n$-vertex $F$-free graphs such that for all $G_1,G_2\in\cG_n$ we have $\cN(H,G_1)=\cN(H,G_2)$. 
Suppose there exists $n_0$ such that for all $n\ge n_0$, every $n$-vertex extremal graph $G$ for $\ex(n,H,F)$ has a subgraph $G'$ with $|V(G')|\le n/2$ that is also a subgraph of some $G_n\in \cG_n$ and 
$\cN_I(H,G,G')\le \cN_I(H,G_n,G')$, with equality only if $G\in\cG_n$.

Then for $n$ large enough, $\ex(n,H,F)=\cN(H,G_n)$ for some $G_n\in\cG_n$. Moreover, every extremal graph belongs to $\cG_n$.

\end{lemma}

Lemma~\ref{progin} follows as corollary from Lemma~\ref{sim-pi}, but can be proved directly. For small $n$, there may be $n$-vertex graphs that have more
copies of $H$ than any $G_n \in \mathcal{G}_n$. However, this means a surplus of only constant number of copies of $H$, and this surplus decreases for $n \geq n_0$ and eventually vanishes. Moreover, this decrease continues and for even larger $n$ a graph not in $\mathcal{G}_n$ cannot be extremal.

\subsection{Algebraic constructions}\label{algebra}

Often in extremal graph theory, constructions are simple like the Tur\'an graph, and the difficult part of a problem is to prove a matching upper bound. In some cases, however, sophisticated constructions, often algebraic in nature, give the best known lower bounds.
Recall Theorem~\ref{fured} that $\ex(n,K_{2,t})=(1+o(1)) \frac{\sqrt{t-1}}{2}n^{3/2}$. Let us describe the construction, due to F\"uredi~\cite{Fur96}, which we term the \emph{F\"uredi graph}.

Let $q_t(n)$ denote the largest prime power such that $t-1$ divides $q_t(n)-1$ and $\frac{q_t(n)^2-1}{t-1} \le n$. It is well-known that for sufficiently large $n$, there exists such a $q_t(n)$ with $\sqrt{nt}-n^{1/3}<q_t(n)$. Therefore, we can construct a graph on $\frac{q_t(n)^2-1}{t-1}$ vertices instead of $n$ vertices, and this does not change the asymptotics of the resulting lower bound.

Consider a $q_t(n)$-element field and let $h$ be an element of order $t-1$. Then the elements $h,h^2,h^3,\dots,h^{t-1}$ form a multiplicative subgroup. Two pairs of elements $(a,b)$ and $(a',b')$ are considered equivalent if $a=h^pa'$ and $b=h^pb'$ for some $p$. This defines $\frac{q_t(n)^2-1}{t-1} $ equivalence classes which will be the vertex set of the F\"uredi graph. The equivalence class of $(a,b)$ is joined by an edge to the equivalence class of $(c,d)$ if and only if $ac+bd=h^p$ for some $p$. Note that this definition allows some vertices to be adjacent to themselves, but these loops are simply ignored. The resulting simple graph is the {\it F\"uredi graph} $F(n,t)$.

F\"uredi~\cite{Fur96} showed that any vertex of $F(n,t)$ has degree $q_t(n)$ or $q_t(n)-1$, and any two vertices have exactly $t-1$ or $t-2$ common neighbors. This completes the proof of the lower bound in Theorem~\ref{fured}. It is not difficult to give an asymptotic count of the number of copies of any tree $T$ in the F\"uredi graph: pick an ordering of the vertices of $T$ such that each vertex but the first has exactly one neighbor before itself. There are $n$ ways to pick the first vertex, and each further vertex is picked as a neighbor of an earlier vertex; there are at least $q_t(n)-|V(T)|$ and at most $q_t(n)$ ways to make this selection. It is somewhat more complicated to count cycles $C_r$. If we pick the vertices $v_1,\dots,v_r$ in a cyclic order, we have to be careful when picking $v_{r-1}$: all but $o(q_t(n))$ of the neighbors of $v_{r-2}$ have $t-1$ common neighbors with $v_1$. 
Gerbner~\cite{Ger21} introduced the following definition: we say that a graph $H$ is {\it $t$-F\"uredi-good} if $\ex(n,H,K_{2,t})=(1+o(1))\cN(H,F(n,t))$. Using this notion, Theorem~\ref{fured} states that $K_2$ is $t$-F\"uredi-good for any $t\ge 2$.
Alon and Shikhelman~\cite{AloShi16} showed that $K_3$ is $t$-F\"uredi-good, Gerbner and Palmer~\cite{GerPal18} showed that $C_r$ and $P_r$ are $t$-F\"uredi-good, and Gerbner~\cite{Ger21} characterized the $t$-F\"uredi-good trees. 

Note that even if $H$ is $t$-F\"uredi-good, its extremal graph may be very different from the F\"uredi graph. Indeed, $P_3$ is $2$-F\"uredi-good, i.e., $\ex(n,P_3,C_4)=(1+o(1))\cN(P_3,F(n,2))=(1+o(1))\frac{1}{2}n^2$, but Gerbner~\cite{Ger20} showed that $\ex(n,P_3,C_4)=\cN(P_3,F^*)$ where $F^*$ is a universal vertex and a matching of size $\lfloor (n-1)/2\rfloor$ on the remaining vertices (e.g.\ for $n$ odd, $F^*$ is a friendship graph).

\smallskip

Let us continue with $K_{s,t}$-free graphs. A lower bound on $\ex(n,K_{s,t})$ for $t>(s-1)!$ is given by an algebraic construction called the \emph{projective norm graph}.
 Let $q$ be a prime power and consider a finite field $\mathbb{F}_{q^{s-1}}$. The \emph{$\mathbb{F}_q$-norm} of an element $A$ is $N(A)=A \cdot A^q\cdot A^{q^2}\cdots A^{q^{t-2}}$. The {\it projective norm graph} $NG(q,s)$ has vertex set $\mathbb{F}_{q^{s-1}}\times \mathbb{F}_{q}^*$, where $\mathbb{F}_{q}^*$ is $\mathbb{F}_{q}$ with 0 removed. Two vertices $(A,a)$ and $(B,b)$ are joined by an edge if $N(A+B)=ab$. Alon, R\'onyai, and Szab\'o~\cite{AloRonSza99} showed that this graph is $K_{s,(s-1)!+1}$-free. There are $q^{s-1}(q-1)$ vertices and, as before, results proved for $NG(q,s)$ give an asymptotic result for general $n$, due to the high density of prime powers among integers.
Every vertex of $NG(q,s)$ has degree $q^{s-1}-1$ or $q^{s-1}-2$, which gives a lower bound on the number of edges. 

Unlike the F\"uredi graph, it is often difficult to estimate the number of copies of certain subgraphs of $NG(q,s)$. For example, it is not known whether $NG(q,s)$ contains a $K_{s,(s-1)!}$. However, as stated in Section~\ref{other}, Alon and Pudl\'ak~\cite{AloPud01} showed the following pseudo-random property of projective norm graphs: any graph $H$ of maximum degree less than $(s+1)/2$, satisfies $\cN(H,NG(q,s))=\Theta(q^{s|V(H)|-e(H)})$.
This property was used by
Alon and Shikhelman~\cite{AloShi16} to give a lower bound on $\ex(n,K_r,K_{s,t})$ when $r\le s+1$ (and of course $t>(s-1)!$). They also gave a lower bound\footnote{We note that the conditions are incorrectly stated in their paper~\cite{AloShi16}.} on $\ex(n,K_{a,b},K_{s,t})$ when $a\le b<(s+1)/2$. Bayer, M\'esz\'aros, R\'onyai, and Szab\'o~\cite{BayMesRon19} extended the pseudo-random property to additional graphs $H$.
Zhang and Ge~\cite{ZhaGe19} used another algebraic construction to prove that $\ex(n,K_r,K_{2,t})=\Theta(n^{3/2})$ if $t\ge 2r-3\ge 3$.

\smallskip

We conclude this subsection by briefly summarizing a
random algebraic construction due to Bukh~\cite{Buk15}. This gives another $K_{s,t}$-free graph with $\Theta(n^{2-1/s})$ edges when $t$ is large enough compared to $s$. The threshold on $t$ is larger than $(s-1)!+1$, so we do not immediately obtain a better bound on $\ex(n,K_{s,t})$ (although a variant of this method was used in~\cite{Buk21} to obtain an improvement). It turns out that this construction is particularly suitable for generalized Tur\'an problems. Ma, Yuan, and Zhang~\cite{MaYuaZha18} showed that the above pseudo-random property extends to every graph $H$ for this construction. More precisely, they showed that for any graph $H$, if $t$ is large enough compared to $s$ and $e(H)$, then $\ex(n,H,K_{s,t})=\Omega(n^{|V(H)|-e(H)/s})$.

Let $f(x_1,\dots,x_s,y_1,\dots,y_s)$ be a symmetric polynomial over $\mathbb{F}_q$ with $2s$ variables and 
let $\cP$ denote the family of such polynomials with degree at most $2d$. Select a polynomial $f$ uniformly randomly from $\cP$ and define
 $G_f$ as the graph with vertex set $\mathbb{F}_q^s$, where $(x_1,\dots, x_s)$ is joined by an edge to $(y_1,\dots,y_s)$ if and only if $f(x_1,\dots,x_s,y_1,\dots,y_s)=0$. From here, the goal is to show that $G_f$ contains many copies of $H$, but few copies of $K_{s,t}$. A key lemma states that for every set $U$ of pairs of vertices, if $|U|$ is small enough compared to $q$ and $s$, then the probability that every pair in $U$ forms an edge of $G_f$ is $q^{-|U|}$. This statement does not depend on the structure of $H$, but it can be used to calculate the expected number of copies of $H$.

Let $f_1,\dots,f_s$ be polynomials of degree at most $d$ in $s$ variables over $\mathbb{F}_q$ and let $W=\{y\in \mathbb{F}_q^s:f_1(y)=f_2(y)=\dots=f_s(y)=0\}$.
Another key lemma states that there exists a $C$ such that $|W|\le C$ or $|W|\ge q-C\sqrt{q}$. Let $S$ be a set of $s$ vertices and consider the polynomials $f(u, \cdot)$ with $u\in S$. Then the common neighborhood of the vertices of $S$ has order at most $C$ or at least $q-C\sqrt{q}$. This gives an upper bound on the number of $s$-sets with common neighborhood of order greater than $C$ (i.e., at least $q-C\sqrt{q}$). From here we can remove a vertex from each such $s$-set to obtain a $K_{s,C+1}$-free graph. Showing that there are still many copies of $H$ completes the probabilistic argument.

\section{Open problems}\label{open}

We conclude this survey with open problems. First, the non-degenerate case, i.e., when $\ex(n,H,F)=\Theta(n^{|V(H)})$. In Section~\ref{diff}, we saw that in this case it is permitted to have graphs $H$ and $F$ such that $\chi(H)\ge \chi(F)$, but we are not aware of any results besides those for $\ex(n,C_{2r+1},C_{2k+1})$ and consequences of Theorem~\ref{genESS}. There are many individual open problems of this flavor. For example, the exact value of $\ex(n,C_5,B_k)$ is unknown (the asymptotics follow from Theorem~\ref{genESS} and the value of $\ex(n,C_5,C_3)$).

Recall that a graph $H$ is $F$-Tur\'an-good if, for $n$ large enough, $\ex(n,H,F)$ is attained by the Tur\'an graph with $\chi(F)-1$ partition classes and $H$ is weakly $F$-Tur\'an-good if, for $n$ large enough, $\ex(n,H,F)$ is attained by an $n$-vertex complete $(\chi(F)-1)$-partite graph. 
The following concrete conjecture on Tur\'an-goodness appears in \cite{GrzGyoSal22} (see Subsection~\ref{good-stable} for more details).

\begin{conjecture}[Grzesik, Gy\H ori, Salia, and Tompkins~\cite{GrzGyoSal22}]
    Any bipartite graph $H$ with diameter at most $4$ is $K_3$-Tur\'an-good.
\end{conjecture}

Let $k(H)$ be the minimum threshold such that if $k \geq k(H)$, then $H$ is $K_k$-Tur\'an good. It is known (see \cite{MorNirNor22}) that $\chi(H)+1<k(H)\le 300|V(H)|^9$ for all $H$. Similar thresholds exist (see \cite{GerHam23}) for the related notions of weakly $K_k$-Tur\'an-good and (weakly) $K_k$-Tur\'an-stable. It would be interesting to give improvements to any of these four thresholds.

The threshold $k(H)$ can be viewed as a statement that among the $n$-vertex $K_k$-free graphs, the graph that maximizes the number of copies of $H$ also maximizes the number of copies of $K_2$, provided that $k$ and $n$ are large enough. One can ask more generally: given $H$, is there a threshold $r(H)$ such that $\Ex(n,H,F)\cap \Ex(n,K_2,F)\neq \emptyset$ provided $\chi(F)>r(H)$ and $n$ is large enough? If this does not hold for every $H$, does it hold for some $H\neq K_2$? Or are there graphs $H\neq H'$ such that $\Ex(n,H,F)\cap \Ex(n,H',F)\neq \emptyset$ provided $\chi(F)$ is above some threshold and $n$ is large enough? Or does this hold if we prescribe some other property of $F$ instead of, or in addition to, the assumption on $\chi(F)$?

There are examples of graphs $H$ with chromatic number $k-1$ that are not weakly $K_k$-Tur\'an-good~\cite{GyoPacSim91,GrzGyoSal22,Ger22d}. However, in such cases, we do not actually know the value of $\ex(n,H,K_k)$, only a construction that contains more copies of $H$ than any $(k-1)$-partite graph. Each of these constructions is an unbalanced blowup of some fixed graph. Is an extremal graph always a blowup of some fixed graph $G_0$? If so, what can we say about $|V(G_0)|$, for example? 

What happens when $F$ is not a clique? For $F$ with a color-critical edge we have: 

\begin{conjecture}[Gerbner~\cite{Ger22d}]
 If $H$ is weakly $K_k$-Tur\'an-good, then $H$ is weakly $F$-Tur\'an-good for every $k$-chromatic graph $F$ with a color-critical edge.     
\end{conjecture}

The above conjecture is true asymptotically using the Removal Lemma, but the converse does not hold. 

The following problem appears in Section~\ref{turangood}:  characterize the pairs of graphs $H,F$ such that
\[
\ex(n,H,F)=\cN(H,K_{\sigma(F)-1}+T(n-\sigma(F)+1,k-1)),
\]
where $\sigma(F)$ is the minimum number of vertices needed to delete from $F$ to obtain a graph with a smaller chromatic number. This property can be seen as a variant of Tur\'an-goodness. By a theorem in \cite{Ger22f}, it is known that for every graph $F$ there exists a graph $H$ that satisfies the equality, but the problem is wide open in general.

\smallskip

Now the degenerate case, i.e., when $\ex(n,H,F) = o(n^{|V(H)|})$.
Alon and Shikhelman~\cite{AloShi16} characterized the graphs $F$ with $\ex(n,K_3,F)=O(n)$ (see Section~\ref{triangle}). There are similar characterizations for graphs besides $K_3$; namely, forests, cycles, and $K_{2,t}$.
What about other graphs, particularly larger cliques? We can also replace linearity with other functions of $n$. For example, which graphs $F$ force $\ex(n,K_3,F)=O(n^2)$ and which force $\ex(n,K_3,F)=o(n^2)$? 
In general, for any $H$, we know the graphs $F$ that demand $\ex(n,H,F)=o(n)$ (which are exactly the graphs $F$ with $\ex(n,H,F)=O(1)$ by results of Gerbner and Methuku~\cite{GerMet25}).

There are examples of graphs $H$ and $F$ for which $\ex(n,H,F)\neq \Theta(n^\alpha)$ for any $\alpha$. For example, in \cite{AloShi16} it was shown that $n^{2-o(1)} \leq \ex(n,K_3,B_k) \leq o(n^2)$ for $k\ge 2$.
Let $T_i$ denote the graph we obtain from $K_3$ by attaching $i$ pendant edges to one of the vertices. Gerbner~\cite{Ger20} showed that $n^{2+i-o(1)}\leq \ex(n,T_i,B_k)\leq o(n^{2+i})$. For which $\alpha$ do we have graphs $H$ and $F$ such that $n^{\alpha-o(1)} \leq \ex(n,H,F) \leq o(n^\alpha)$?
Recall the following general conjecture from Subsection~\ref{diff}.

\rationalconj*

There are conjectures for classical Tur\'an numbers where the analogous statements in the generalized setting might be easier to prove. 
One example is the \emph{rational exponents conjecture} (and its generalization above), another is the Erd\H os-S\'os conjecture. It was shown in~\cite{GerMetPal18} that if each subtree of a $k$-vertex tree $T$ satisfies the  Erd\H os-S\'os conjecture (Conjecture~\ref{erdsos}), then $\ex(n,K_r,T)\le \frac{1}{k-1}\binom{k-1}{r}n$. However, since one of the conjectured nearly extremal graphs for $\ex(n,T)$ is $\lfloor \frac{n}{k-1}\rfloor K_{k-1}$, which contains many copies of $K_r$, it might actually be easier to prove the bound for $r>2$. Indeed, for $r=k-1$ it is trivial. We make the following conjecture.

\begin{conjecture}
    Let $T$ be a $k$-vertex tree let $r\ge 3$. For $n=a(k-1)+b$ with $b<k-1$, 
    \[
    \ex(n,K_r,T)=\cN(K_r,aK_{k-1}\cup K_b).
    \]
\end{conjecture}

We now consider problems when we fix one of the graphs, but the other can be (almost) any graph. What can we say in this case? When the counting graph $H$ is fixed, we do not have a full characterization of the order of magnitude, even if $H$ is a single edge. However, for some graphs $H$ we may be able to say something about the relationship between $\ex(n,H,F)$ and $\ex(n,F)$. For $H$ a matching $M_r$, we have $\ex(n,M_r,F)=\Theta(\ex(n,F)^r)$, moreover, $\ex(n,M_r,F)=(1+o(1))\frac{1}{r!}\ex(n,F)^r$ if $F$ is a tree or not a forest, due to~\cite{Ger20,Ger24d}. For $H=S_r$, the order of magnitude of $\ex(n,S_r,F)$ was almost completely determined in~\cite{FurKun06}. What about other graphs $H$?

Consider now the case where $F$ is fixed. If $F=K_r$ and $H$ is $K_r$-free, then blowups of $H$ are also $K_r$-free, thus by Proposition~\ref{degen-prop} we have that $\ex(n,H,F)=\Theta(n^{|V(H)|})$. The case $F=tK_r$ is more complicated, but Gerbner~\cite{Ger23b} determined the order of magnitude of $\ex(n,H,tK_r)$ for every $H$. Letzter~\cite{Let18} obtained a similar characterization for trees $F$. What about other graphs $F$?

There are other phenomena worth investigating. Recall that Chase~\cite{Cha19} showed that $\ex(n,K_r,S_k)$ is given by vertex-disjoint copies of $K_{k-1}$ and a smaller clique, assuming $r>2$. Depending on $n$ and $r$, it is possible that the remaining smaller clique has order less than $r$, thus there may exist vertices and edges that are not in any copy of $K_r$. More generally, when do we have that a graph in $\Ex(n,H,F)$ contains some \textit{useless} vertices or edges, i.e., vertices or edges that are not contained in any copy of $H$? Even more generally, we can examine useless copies of some graph $H'$, i.e., copies of $H'$ that do not share any vertex with any copy of $H$.

Finally, let us state more specific questions. Some of the fundamental problems regarding $\ex(n,H,F)$ are when $H$ and $F$ come from the same family of graphs. If $H$ and $F$ are both cliques, then Zykov's theorem determines $\ex(n,H,F)$ exactly. If $H$ and $F$ are both cycles, the order of magnitude of $\ex(n,H,F)$ is determined in~\cite{GisSha18} unless $H=K_3$. As we have described in Subsection~\ref{cycles}, $\ex(n,K_3,C_5)=\Theta(n^{3/2})$, and the upper bound, after several papers, is quite close to the lower bound. We know that $C_{2r}$ is $C_{2k+1}$-Tur\'an-good~\cite{Ger20c} which gives an exact result. There are other exact and asymptotic results for $\ex(n,C_r,C_k)$ when $k$ or $r$ are small. Perhaps the exact value of $\ex(n,C_6,C_8)$ is the natural next question. Recall that $\ex_{\textrm{bip}}(n,C_6,C_8)$ was determined in~\cite{GyoHeLv22}.

Just like cycles, the first open case for paths is $\ex(n,P_6,P_8)$.
The asymptotics of $\ex(n,P_r,P_k)$ were determined in~\cite{GyoSalTom18}. Moreover, they found the exact value in several cases and, for $k \geq k$ odd, they conjectured the extremal graph for $\ex(n,P_{k-1},P_k)$. Let us describe their construction. Begin with a clique $K_{\lfloor k/2\rfloor }$ and identify one of its vertices $v$. Add $m$ new independent vertices each adjacent to $v$ and then $n-m-\lfloor k/2 \rfloor$ further independent vertices adjacent to all vertices of the clique $K_{\lfloor k/2\rfloor}$ except for $v$. The number of copies of $P_{k-1}$ depends on $m$ and so let $H_{n,k}$ be such a graph that maximizes the number of copies of $P_{k-1}$.

\begin{conjecture}[Gy\H ori, Salia, Tompkins, and Zamora~\cite{GyoSalTom18}]
    For $k \geq 5$ odd, the extremal number $\ex(n,P_{k-1},P_k)$ is attained by the graph $H_{n,k}$.
\end{conjecture}

For stars and matchings, the problem is trivial. For $\ex(n,K_{a,b},K_{s,t})$, there are several open cases. We highlight the case $1<s<a\le b<t$, where knowledge is the most limited. In this range, we have the upper bound $O(n^s)$  and there are several lower bounds in different ranges of $a,b,s,t$, each of order of magnitude at most $n^{3/2}$.

We do not make an exhaustive list the remaining open problems on $\ex(n,H,F)$ for $H$ and $F$ from distinct graph classes where there are no prior results. 
However, we note that there are hardly any results for $\ex(n,K_{s,t},C_{2k})$ if $s,t> 2$, for $\ex(n,C_r,K_{s,t})$ (even in the case $s=1$), for $\ex(n,C_r,K_k)$ for $k>3$, or for $\ex(n,P_r,K_{s,t})$ with $s>2$.

\bibliography{gts}
\bibliographystyle{abbrvurl}

\end{document}